\documentclass{article}
\usepackage[utf8]{inputenc}
\usepackage{amsfonts}
\usepackage{amsmath}
\usepackage{amssymb}
\usepackage{amsthm}
\usepackage{hyperref}
\usepackage{color}
\usepackage{tikz}
\usepackage{dsfont} %pour l'indicatrice
\usepackage{stmaryrd} % pour \llbracket n~;~ p \rrbracket
\allowdisplaybreaks %pour que les align* puissent se faire sur plusieurs pages et ne pas laisser de "blancs"

\usepackage[textwidth=2.5cm]{todonotes}
\presetkeys{todonotes}{size=\scriptsize}{}

\usepackage{mathrsfs} %pour un beau L ronde

\usepackage{comment} % pour les "%"
\usepackage{xcolor}
\usepackage{bbding,pifont} %pour l'astérisque du beau schéma
\usepackage{float} %pour mettre bien une image?
\usepackage[nottoc,numbib]{tocbibind} %pour avoir les réf biblio en mode [Truc97] et pas [1]
\newcommand{\R}{\mathbb{R}}
\newcommand{\E}{\mathbb{E}}
\newcommand{\lr}[1]{\left( #1 \right)}
\newcommand{\lrcr}[1]{\left[ #1 \right]} %crochets
\newcommand{\lrac}[1]{\left\{ #1 \right\}} %accolades
\newcommand{\lrva}[1]{\left| #1 \right|} %|.|
\newcommand{\lrdb}[1]{\left\| #1 \right\|} %||.||
\newtheorem{theorem}{Theorem}

\title{Stochastic Modelling of Telomere Length Dynamics During Hematopoiesis}
\author{Sophie Baland$^{1}$, Simon Toupance$^{2}$ and Denis Villemonais$^{3,4}$}
\begin{document}
	
	\maketitle
	
\footnotetext[1]{Universit\'e de Lorraine, CNRS, Inria, IECL, F-54000 Nancy, France \\
  E-mail: sophie.baland@univ-lorraine.fr}

\footnotetext[2]{Université de Lorraine, Inserm, DCAC, F-54000 Nancy, France \\
  E-mail: simon.toupance@univ-lorraine.fr}

  \footnotetext[3]{Université de Strasbourg,  IRMA, Strasbourg, France\\ E-mail: denis.villemonais@unistra.fr}
  
  \footnotetext[4]{Institut Universitaire de France}

	\begin{abstract}
		We introduce a new stochastic model of telomere length dynamics during hematopoiesis. The numerous cell divisions occurring in this setting makes telomere length dynamics, and in particular their attrition, a crucial aspect of hematopoiesis. After recalling the main aspect of the biological setting, we describe an individual based model of hematopoietic population cells. We then study its behaviour in the large population limit, as well as its fluctuations.
	\end{abstract}
	
	\section{Introduction}

	Mathematical modeling of cellular development has become a fundamental tool in biology and medicine for analyzing and understanding complex biological systems. Advances in quantitative biology and computational methods have enabled the development of models capable of accurately reproducing cellular dynamics and predicting their evolution  \cite{KitanoHiroaki2002SBAB,tomlin2007biology, glynn2014mathematical,voit2017first, alon2019introduction}.

	Telomeres are nucleoprotein structures located at the ends of chromosomes, consisting of repeated TTAGGG sequences, which ensure the protection of genomic integrity. At each cell division, their length progressively decreases due to the end-replication mechanism. When a critical threshold is reached, cells enter senescence or undergo apoptosis. Telomere shortening is influenced by both genetic determinants present at birth and environmental factors acting during postnatal life, as highlighted in~\cite{Benetos2022}. This process is closely associated with cellular ageing, the onset of degenerative diseases, and reduced life expectancy
	~\cite{blackburn2000telomere, aviv2006telomeres}.
	In certain cell types, telomere shortening can be partially compensated by the action of telomerase, an enzyme capable of elongating chromosomal ends. This ribonucleoprotein enzyme uses an intrinsic RNA template to synthesize new telomeric sequences at the end of the DNA strand, after which conventional DNA replication mechanisms restore the double-stranded structure~\cite[Chapter~6]{alberts2015essential}.

	A substantial body of mathematical work has been devoted to understanding telomere length dynamics in humans, see e.g.\cite{levy1992telomere,arino1995mathematical,olofsson1999stochastic,portugal2008computational} for seminal works on telomere shortening modelisation, \cite{ rat2023dynamiques,qi2014stochastic, benetos2024stochastic,olaye2024long,benetos2025branching} for stochastic processes and \cite{hirt2014modelling,wattis2020mathematical, olaye2025inverse} for deterministic models. Similarly, several biological works have focused on telomere length dynamic during hematopoiesis~\cite{prade2007regulation, zimmermann2004effects, greenwood2003telomeres}.
	In particular, the stochastic branching model introduced in~\cite{benetos2024stochastic,olaye2024long} provides a relevant probabilistic framework that explicitly incorporates the role of telomerase in these dynamics. Note that, since cell division rates are heterogeneous across the organism, telomere length depends on the number of divisions, telomerase activity, and the nature of the cells under consideration~\cite{Benetos2022}. Hematopoietic stem cells constitute a particularly relevant system due to their high proliferative capacity and the specific dynamics of their telomeres. In this context, it is natural to investigate telomere length evolution within these cells, as well as more broadly during hematopoiesis, defined as the set of processes responsible for blood cell production. 
	
	The objective of the present work is to propose a continuous-time stochastic branching model to describe the evolution of cell populations during hematopoiesis. This model explicitly incorporates telomere length at the single-cell level, telomere attrition, and telomerase activity in specific cell populations. While several mathematical models have been proposed to describe the temporal evolution of cell populations during hematopoiesis under both physiological and stress conditions, notably in~\cite{adimy2008modelisation,Adimy_2005, roeder2002novel}, to the best of our knowledge, the present work is the first model to integrate all these mechanisms simultaneously within a unified framework.

	In addition to the telomere dynamics, our paper differs from existing stochastic hematopoiesis models, since we consider a model with several compartments to emulate the existence of several cells type  in the haematopoietic hierarchy, where several studies in the literature model hematopoiesis using a limited number of compartments. For instance, in~\cite{adimy2008discrete}, the authors propose a model based on a system of partial differential equations distinguishing two cell populations: immature and mature cells.
	In~\cite{bonnet2021large,bansaye2026hematopoiesis}, the authors investigate erythropoiesis by considering three distinct compartments: hematopoietic stem cells, progenitors, and red blood cells.%\denis{et le papier récent avec un continuum?}\sophieIL{en + de le citer, dois-je développer davantage ce papier?}

	Compared to the branching cell population model with  telomere dynamics proposed in~\cite{benetos2024stochastic,olaye2024long}, the main difference is that cells are not assumed to be homogeneous, since their behavior depends on the compartment to which they belong: for instance, telomerase activity and cell division rates are allowed to vary across compartments. In addition, the underlying state space is inherently reducible, while the irreducibility of the state space played an important role in the long time behavior of the models of both previously cited papers.

		The paper is organised as follows. In Section~...., we present the main biological mechanisms which underpin the construction of our model. In Section~...., we describe in details our stochastic model for the telomere length dynamics during haematopoiesis.
				Once the model is established, we pursue in Section~.... with the study of its mathematical properties. We first show that, in a regime with a large number of cells, the overall behavior of both the cell population and telomere lengths can be approximated by a deterministic dynamics, despite the intrinsically stochastic nature of the underlying mechanisms. To this end, we establish a law of large numbers in a Skorokhod space, characterizing the asymptotic behavior of the process as the initial number of cells tends to infinity, under an appropriate rescaling. We then investigate the stochastic fluctuations of the model around this deterministic approximation. For this purpose, we prove a central limit theorem–type result, again relying on tightness properties in a Skorokhod space and on the characterization of the limiting process through a martingale problem.

		\section{DNA replication, telomere length dynamic and hematopoiesis}
		
		\subsection{Replication forks}

		DNA replication is a fundamental process ensuring the accurate transmission of genetic information prior to cell division and relies on the coordinated action of specialized proteins. Note that DNA repair mechanisms, although essential, are not addressed here. The reader is referred to~\cite{alberts2015essential} and~\cite{carroll2013introduction} for further details.

		\begin{figure}
			\centering
			\includegraphics[width=0.7\textwidth]{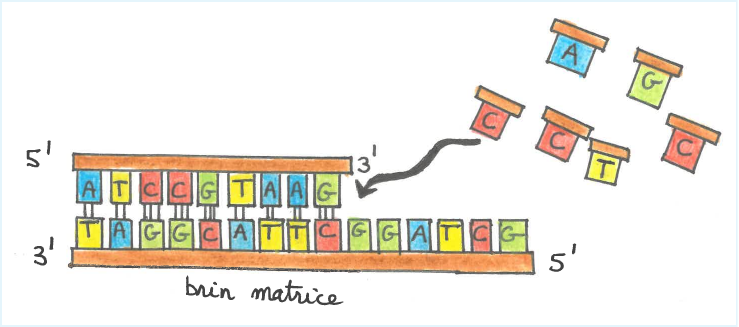}
			\caption{The complementary strand is synthesized using the parental strand as a template. %{\color{gray}Le brin complémentaire se forme à partir du brin parental qui sert de matrice.} 
				(adapted from~\cite{alberts2015essential}, Chapter 6)} 
			\label{figure1}
		\end{figure}

		\begin{figure}
			\centering
			\includegraphics[width=1\textwidth]{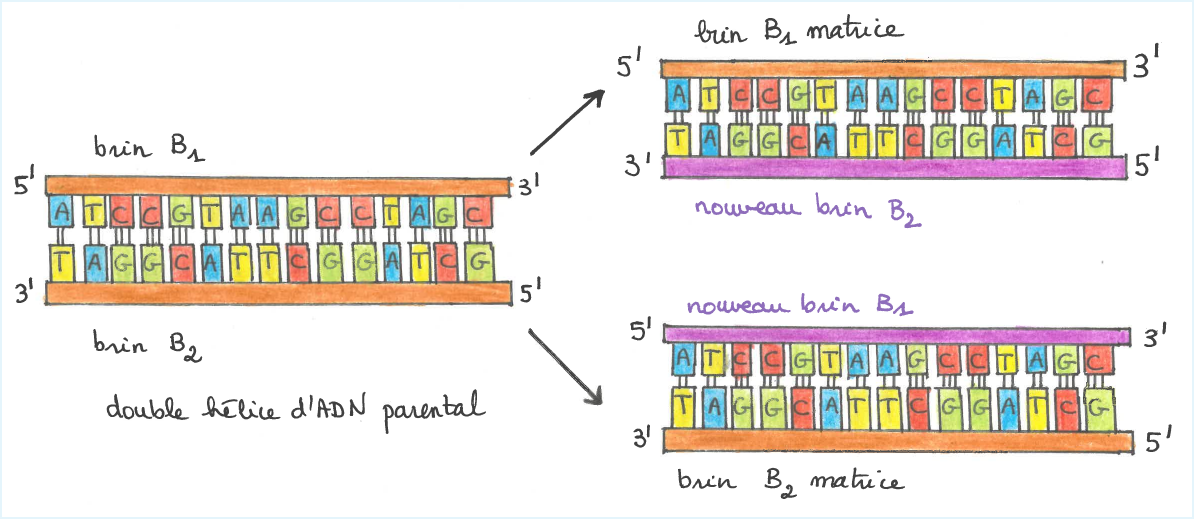}
			\caption{The two strands of the parental helix separate (adapted from~\cite{alberts2015essential}, Chapter 6).} 
			\label{figure2}
		\end{figure}

		During replication, each parental DNA strand serves as a template for the synthesis of a complementary strand, resulting in two identical DNA molecules following strand separation (see Figures~\ref{figure1} and~\ref{figure2}). This process, driven by a highly efficient replication machinery, enables the accurate copying of billions of nucleotide pairs within a limited time, thereby ensuring faithful transmission of genetic information.

		Figure~\ref{figure3} illustrates successive cycles of DNA replication, during which each cell division produces two double helices identical to the parental molecule. This process is referred to as semi-conservative, as each daughter molecule consists of one parental strand and one newly synthesized strand.
		
		\begin{figure}
			\centering
			\includegraphics[width=0.9\textwidth]{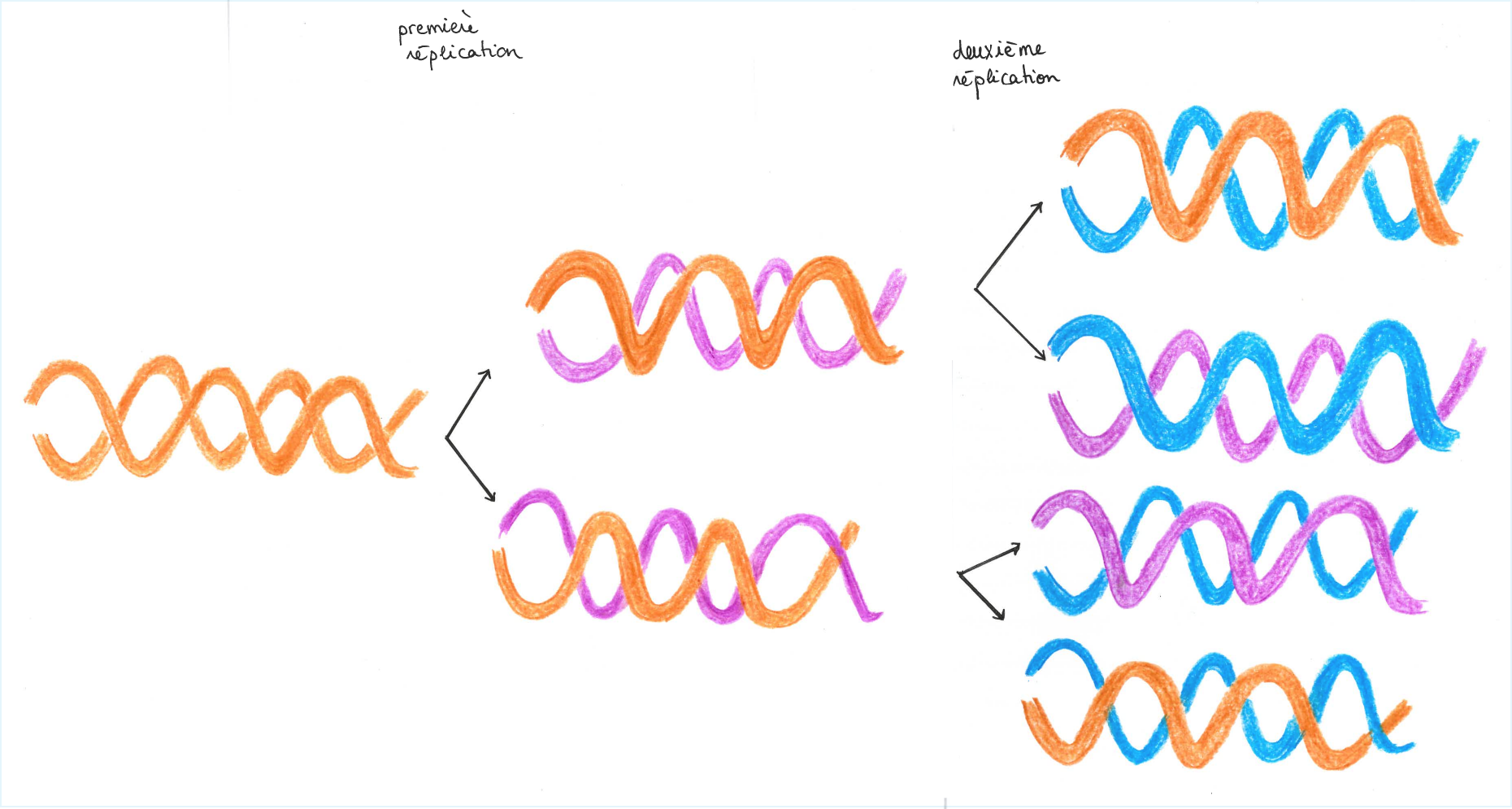}
			\caption{Two successive rounds of replication from two parental DNA strands (adapted from~\cite{alberts2015essential}, Chapter 6).} 
			\label{figure3}
		\end{figure}

		The DNA double helix is stabilized by hydrogen bonds between its two strands, whose their separation is required to allow each strand to serve as a template during replication. This process initiates at specific DNA sequences, known as origins of replication, where initiation proteins recognize the DNA and induce local unwinding of the double helix (see Figure~\ref{figure4}). Specialized enzymes, called helicases, are responsible for breaking hydrogen bonds and promoting strand separation. These regions, often enriched in A–T base pairs, which are easier to dissociate than C–G pairs, constitute preferential sites for replication initiation. In eukaryotic organisms, particularly in humans, multiple origins of replication are activated simultaneously to optimize the duration of the process. Once initiation has occurred, a coordinated set of proteins ensures the proper progression of the replication machinery.

		\begin{figure}
			\centering
			\includegraphics[width=0.5\textwidth]{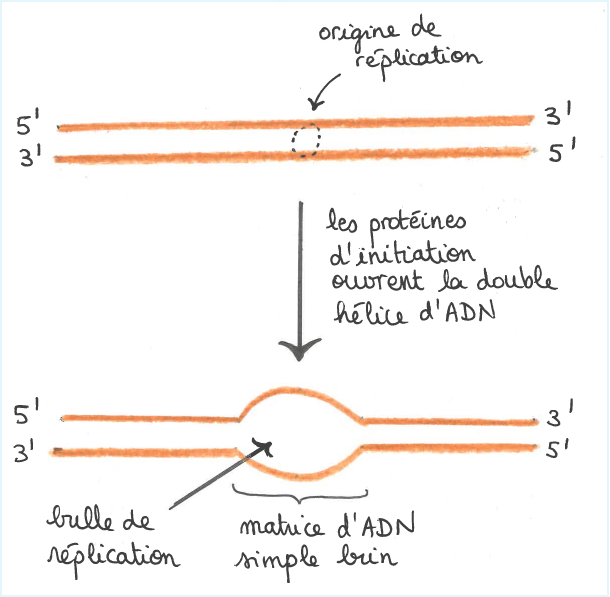}
			\caption{The origin of replication enables the opening of the DNA double helix (adapted from~\cite{alberts2015essential}, Chapter 6).} 
			\label{figure4}
		\end{figure}

		Replication forks are Y-shaped structures that form during local unwinding of the DNA double helix, where parental strands serve as templates for the synthesis of complementary strands. Originating from replication origins, they progress in opposite directions, reflecting bidirectional replication, with a speed of approximately 100 base pairs per second in humans. 

		Note that, due to the directionality of DNA polymerase, replication at the fork is inherently asymmetric: there are two ends called $3'$ and $5'$, and  DNA polymerase synthesizes exclusively in the  $5'$
		to  $3'$
		direction, . As a result, one strand is synthesized continuously (leading strand), whereas the other is produced discontinuously as Okazaki fragments (see  Figure~\ref{figure8}) that are later joined together (lagging strand).
		
		%htbp
		\begin{figure}
			\centering
			\includegraphics[width=1\textwidth]{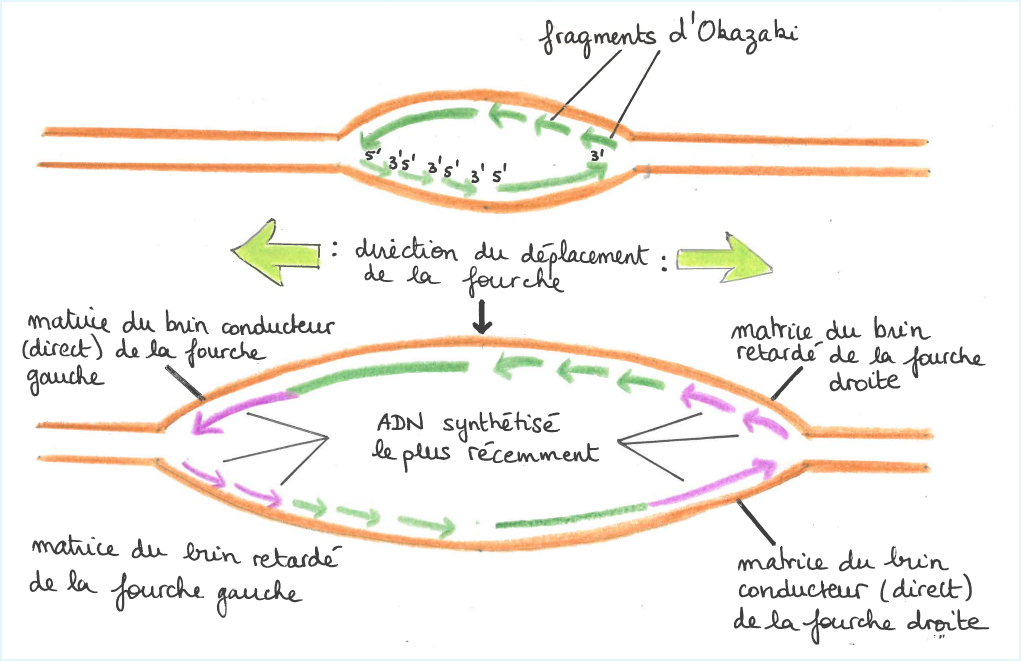}
			\caption{Asymetry of the replication fork. (adapted from~\cite{alberts2015essential}, Chapter~6)} 
			\label{figure8}
		\end{figure}

		\subsection{The end-replication problem and telomerase}
		\label{soussectelomerase}

		Telomeres play a fundamental role in maintaining genome stability in eukaryotic cells, as they condition replicative capacity and influence cellular aging and transformation processes. Their progressive shortening, primarily driven by the end-replication problem identified in the 1970s 
			\cite{watson1972origin, olovnikov1971principle, olovnikov1973theory}, constitutes a central biological phenomenon associated with senescence, apoptosis, and various pathologies.
			The progressive shortening is due to the fact that
			the DNA replication machinery is unable to fully replicate chromosomal ends, leading to a gradual loss of telomeric sequences at each cell division. 
			
			To understand the end-replication problem, we observe that telomeres  exhibit a single-stranded $3'$ overhang, rich in guanine, whose length is variable~\cite{grach2013telomere}. During telomere replication (see Figure~\ref{figure9}), lagging-strand synthesis remains incomplete due to the inability to fully replace the terminal RNA primer, resulting in shortening of the newly synthesized strand. Following replication, daughter molecules display asymmetric ends, with one carrying a 
			$3'$ overhang and the other a blunt end. To restore a consistent structure, nucleases process the blunt ends to generate new 
			$3'$
			overhangs, therefore shortening the  previously existing $5'$ end. The length of these overhangs thus contributes to progressive telomere shortening over successive cell divisions. 
			
			%htbp
			\begin{figure}
				\centering
				\includegraphics[width=1.12\textwidth]{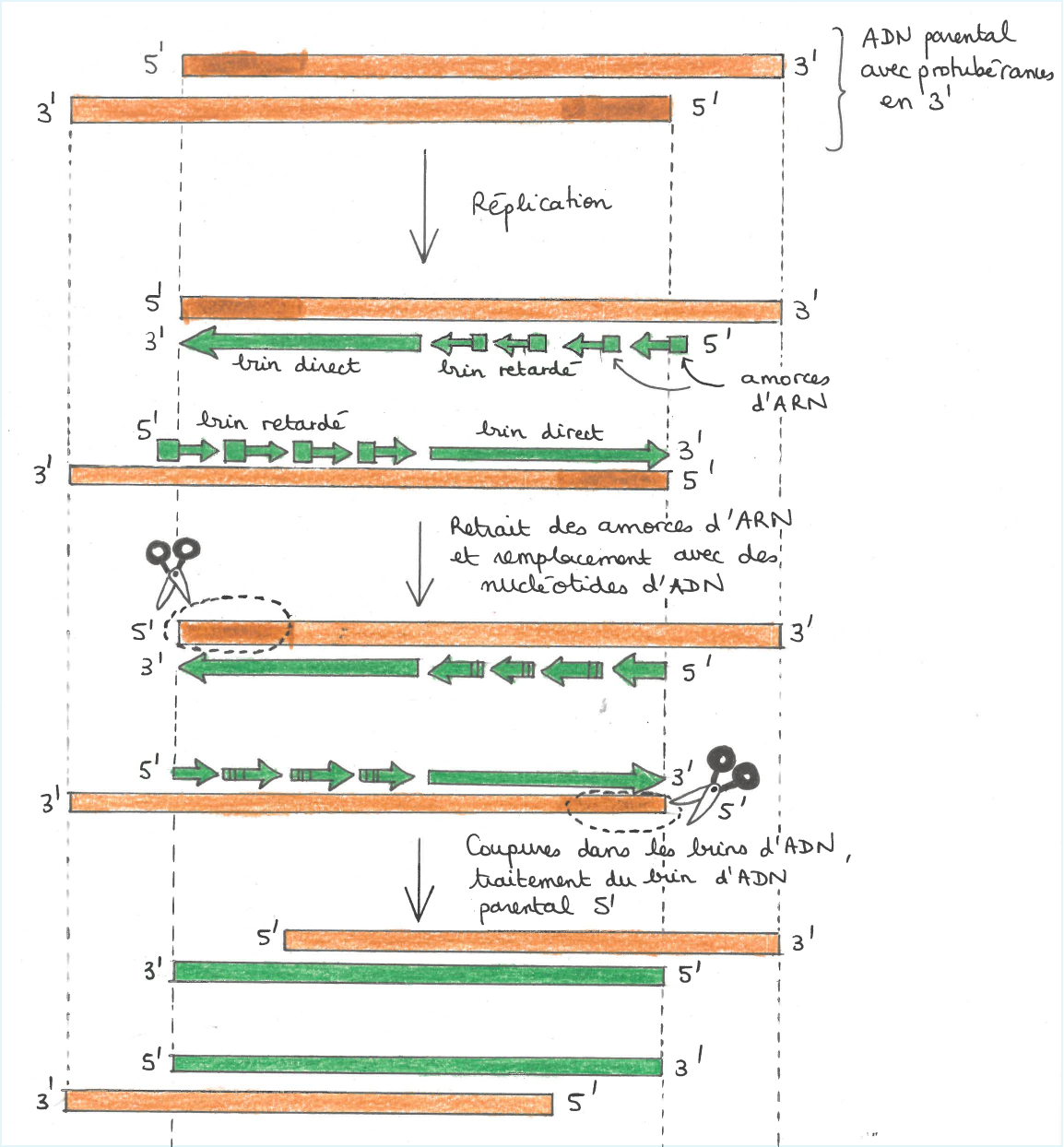}
				\caption{The end-replication problem (adapted from~\cite{grach2013telomere})} 
				\label{figure9}
			\end{figure}

			\begin{figure}
				\centering
				\includegraphics[width=1\textwidth]{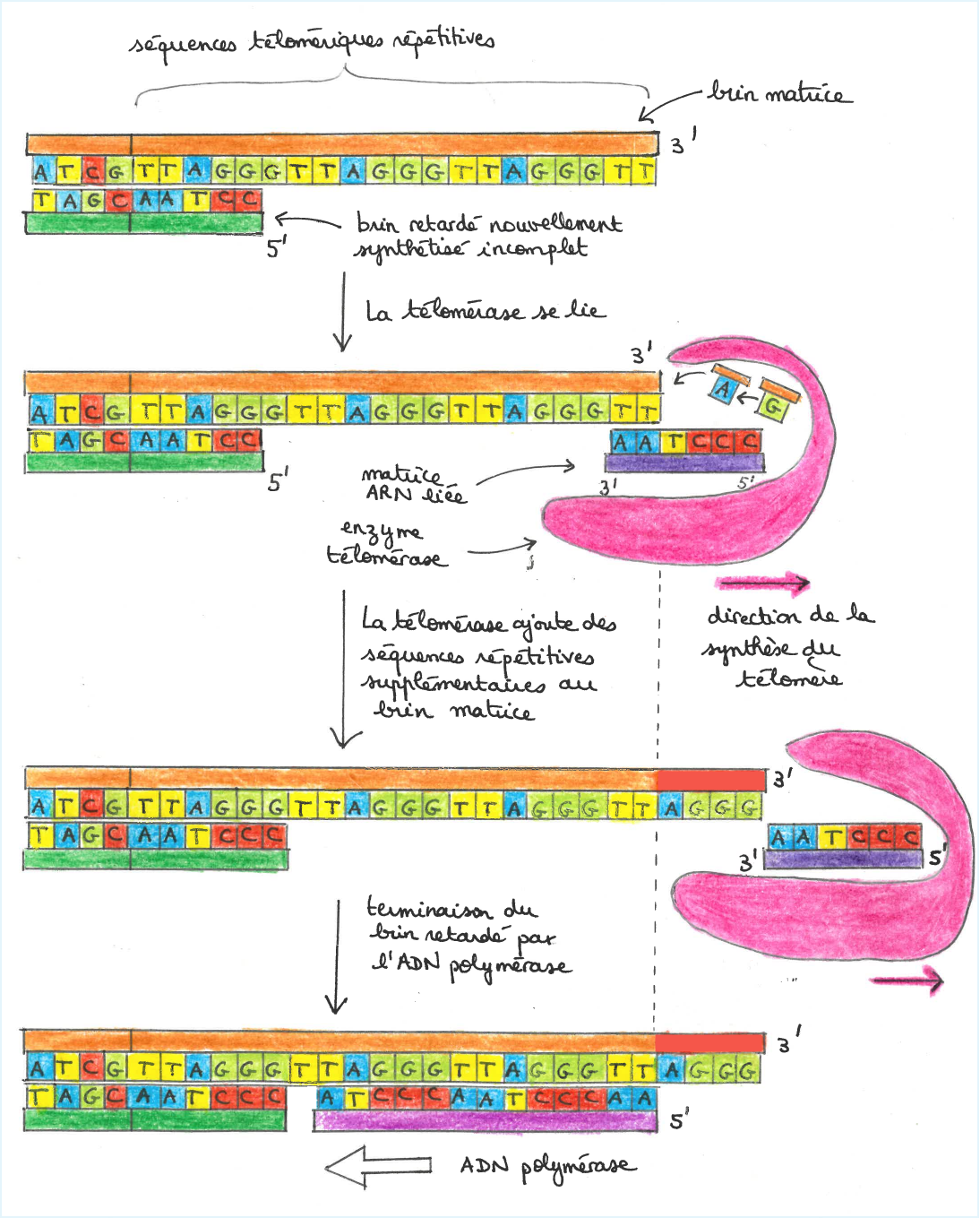}
				\caption{The mechanism of action of the telomerase enzyme (adapted from~\cite{alberts2015essential}, Chapter 6).} 
				\label{figure10}
			\end{figure}

			The inability to fully replicate chromosomal ends leaves an unreplicated region after removal of RNA primers. Without a compensatory mechanism, this limitation lead to progressive telomere shortening at each cell division. This effect is, in some cell types, counteracted by telomerase, an enzyme that binds to chromosome ends and, using an internal RNA template, synthesizes telomeric repeats at the $3'$ end of the parental strand (see Figure~\ref{figure10}). This extension subsequently enables complete replication of the lagging strand through conventional mechanisms. Telomerase was discovered in 1985 by Elizabeth Blackburn, Carol Greider, and Jack Szostak, who were awarded the Nobel Prize in Physiology or Medicine in 2009. 
			We emphasize that telomerase is not active in all cell types, and is actually not active in most cell tissues~\cite{shay2019telomeres,kim1994specific, shay1997survey}, but is strongly expressed in germinal cells, stem cells (and particularly in hematopoietic stem cells), some cells from the immune system~\cite{blackburn2000telomere,ju2006telomeres, shay1997survey} and most cancer cells~\cite{shay1997survey}.

			\subsection{Hematopoiesis}

			Hematopoiesis refers to the process by which blood cells are produced, including leukocytes, erythrocytes, and platelets, with a daily production exceeding $4.4\times 10^{11}$
			cells with varying lifespans. All these cells originate from hematopoietic stem cells (HSCs), primarily located in the bone marrow, which can either self-renew through symmetric division or differentiate through asymmetric division.

			\begin{figure}
				\centering
				\includegraphics[width=0.8\textwidth]{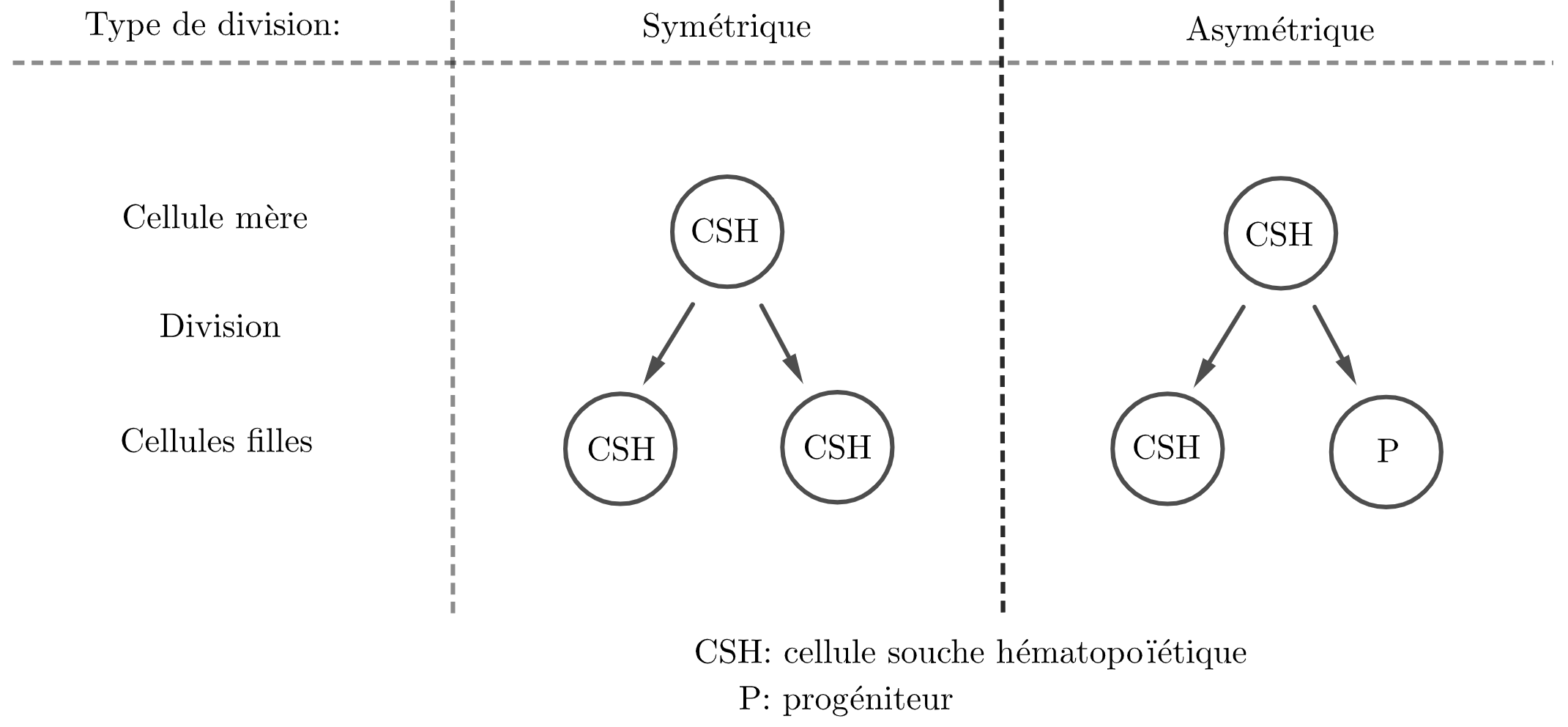}
				\caption{Different CSH division types} \label{celldivSyAsy}
			\end{figure}

			The study of HSCs remains challenging due to their low abundance and the experimental constraints related to in vitro culture. On top of hematopoiesis hierarchy, HSCs are multipotent and can differentiate into either a common lymphoid progenitor or a common myeloid progenitor (see Figure~\eqref{schemahematopogeogebra}), which subsequently commit to specific lineages producing lymphocytes (B, T, NK, and certain dendritic cells) and the precursors of erythrocytes, platelets, and various leukocytes, respectively \cite{kindt2007kuby}. This differentiation process is regulated by extracellular signals, particularly growth factors, and dendritic cells (which play a pivotal role in the innate detection of pathogens and the subsequent activation of the adaptive immune response) may arise from both lineages.
			
			\begin{figure}[htbp]
				\centering
				\includegraphics[width=1\textwidth]{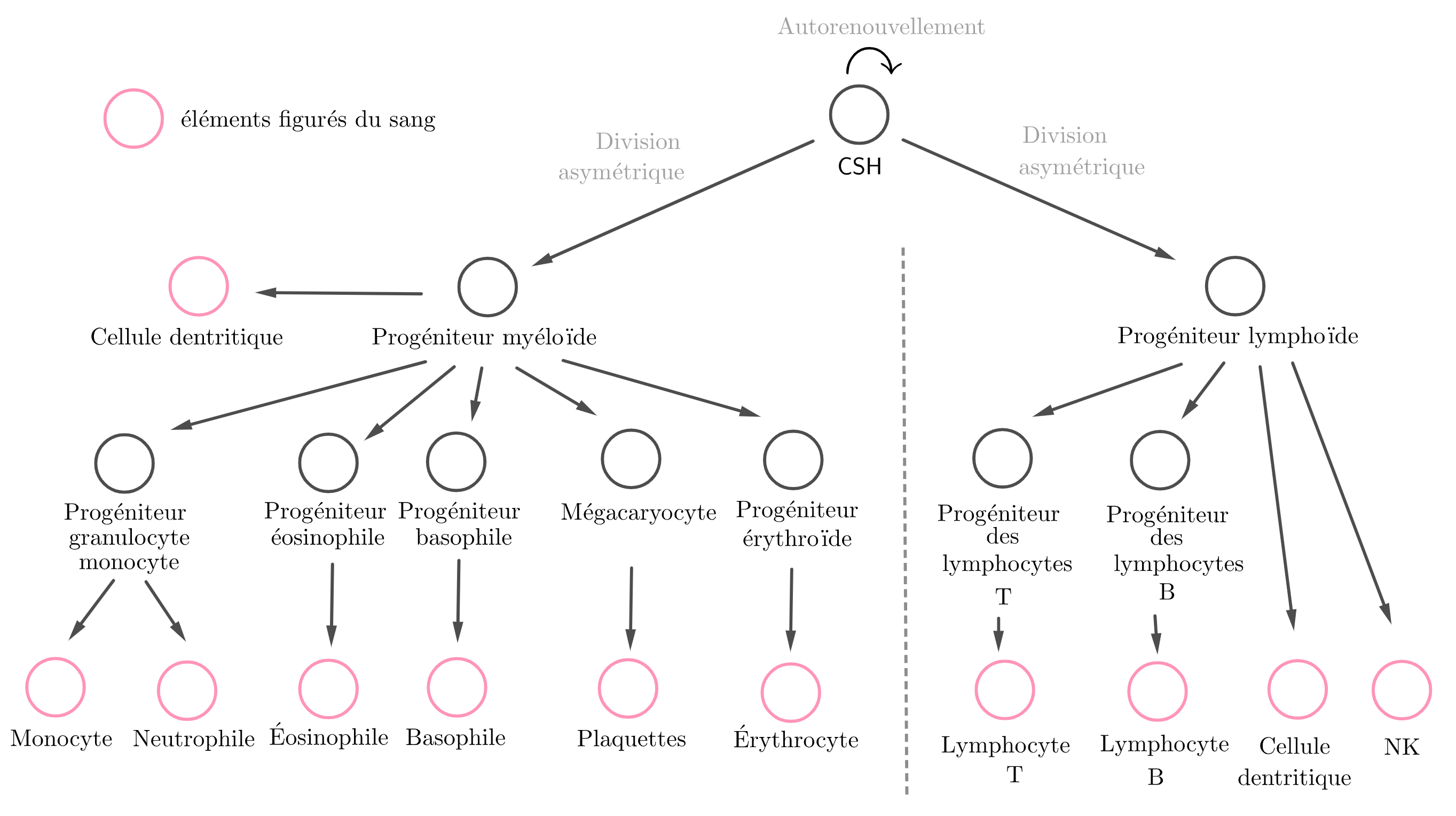}
				\caption{Hematopoiesis hierarchy (adapted from~\cite{kindt2007kuby}, Chapter~2)}
				\label{schemahematopogeogebra}
			\end{figure}

			The self-renewal capacity of cell populations other than hematopoietic stem cells remains debated and lacks consensus. Some studies suggest that progenitor cells retain self-renewal ability~\cite{adimy2008modelisation}, whereas others highlight heterogeneity within the progenitor compartment, including long-term and short-term stem cells as well as non-self-renewing multipotent progenitors~\cite{cheshier1999vivo}. Additionally, certain models assume that both immature and mature cells can self-renew~\cite{adimy2008discrete}, while other works associate this property with stem cell quiescence~\cite{laurenti2018haematopoietic}. In this context, the precise definition of the progenitor compartment is crucial. In our model, we adopt the assumption that self-renewal capacity decreases during differentiation and progressively vanishes as cells commit to specialized lineages~\cite{orkin2008hematopoiesis}.
			
			\subsection{Hematopoïesis, telomere and telomerase}

			We focus here on the dynamics of telomere length during hematopoiesis, based on~\cite{prade2007regulation}. At each cell division, telomeres undergo progressive shortening, which is partially compensated by telomerase activity that adds nucleotide repeats to chromosome ends, thereby limiting attrition and contributing to the preservation of genomic integrity (see Section~\ref{soussectelomerase}).

			The activation of telomerase during hematopoiesis, similarly to self-renewal capacity, remains a matter of debate. Some studies suggest high telomerase activity in progenitor cells due to their rapid proliferation~\cite{prade2007regulation}, whereas others report low activity in these same populations~\cite{zimmermann2004effects}. Moreover, the presence of telomerase in hematopoietic stem cells remains unclear and appears to be generally limited in HSCs, progenitors, and lymphocytes~\cite{greenwood2003telomeres}. In any case, its activity does not seem sufficient to fully prevent telomere attrition, the precise regulatory mechanisms of which during hematopoiesis remain poorly understood. It is generally accepted that telomerase primarily mitigates this shortening by contributing to cell proliferation, protecting chromosomal ends from DNA damage, and preventing cellular senescence, rather than completely inhibiting telomere shortening~\cite{prade2007regulation}. 
			In our model, described in the next section, we assume that telomerase activates in HSCs only.

			\section{A multi-type branching process for the modelisation of hematopoiesis and associated telomere dynamics}

			In this section, we propose a branching process to model the evolution of a population of cells and their telomere lengths during  hematopoiesis, involving several factors such as self-renewal, specialization, telomere attrition, telomerase, differentiation, and death. As noted in the previous section, our biological assumption is that telomerase is active only in HSCs.

			We consider a model with three types of cells, distributed across L compartments
			\begin{itemize}
				\item \textit{HSCs}, which are contained in compartment 1,
				\item \textit{progenitor cells}, which are contained in the $L-2$ compartments numbered from 2 to $L-1$,
				\item \textit{lymphocytes}, which are contained in compartment $L$.
			\end{itemize}

			\begin{figure}[htbp]
				\centering
				\includegraphics[width=0.8\textwidth]{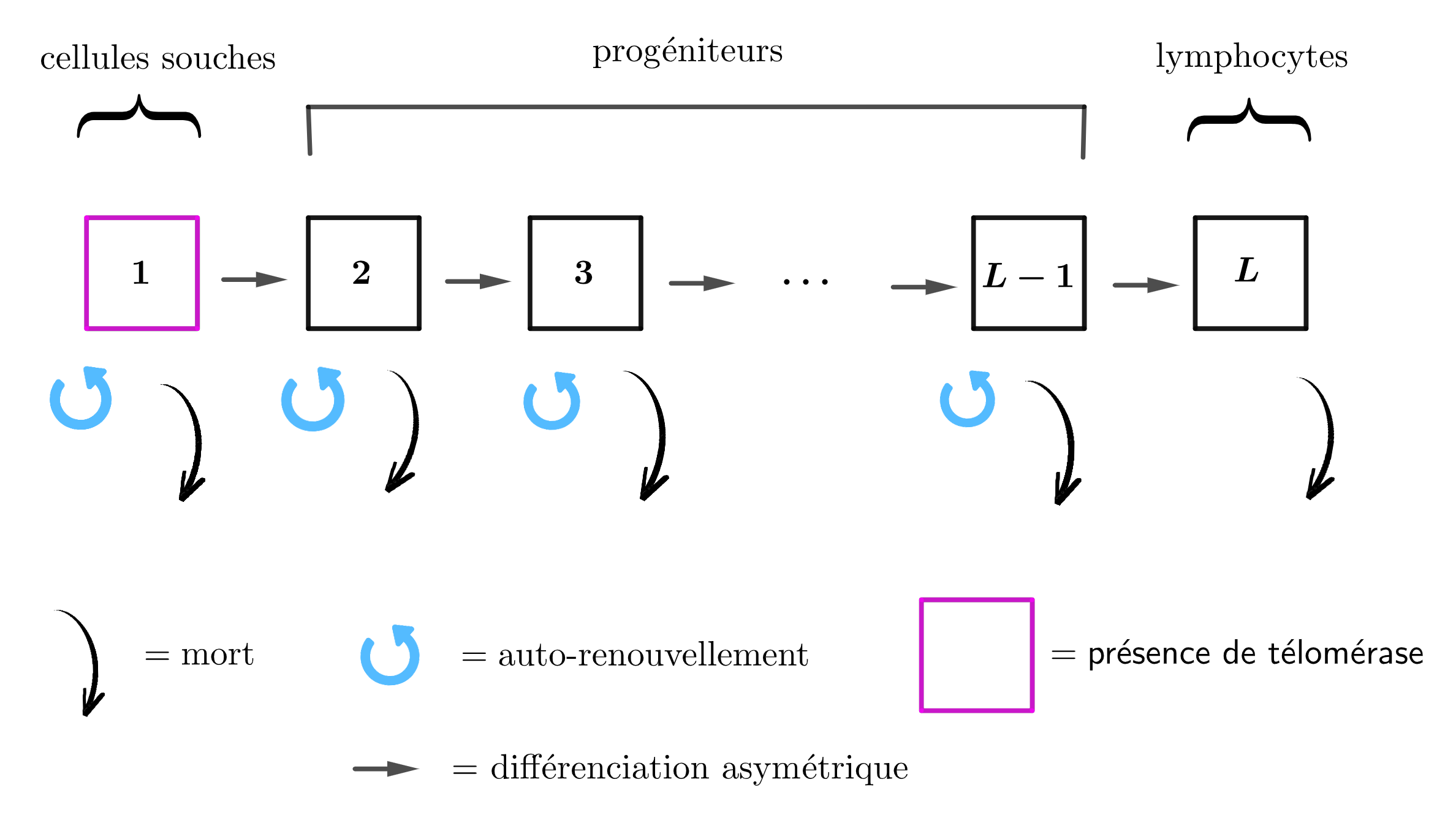}
				\caption{Schematic representation of one lineage in hematopoiesis}
			\end{figure}

			As explained in subsection~\ref{soussectelomerase}, the $3'$ end of telomeres is unpaired and extends beyond the $5'$ end by a length $\omega$ (see Figure~\ref{figure11}), which does not vary from one telomere to another. In particular, the length of a telomere is entirely characterized by the length of its $3'$ strand. The \textit{configuration} of a chromosome will hereafter be denoted by a pair $(m,n)$, where $m$ and $n$ are the lengths of the two $3'$ ends of that chromosome 
			($m$ denotes the length of the $3'$ end of the first strand, and $n$ denotes the length of the $3'$ end of the second strand). We let $\ell_{\text{max}}\in\mathbb N$ denote the maximum length of a telomere, so that the length of a telomere is an element of $\left\{ 0, \cdots, \ell_{\text{max}} \right\} $, and the configuration of a chromosome is an element of $\left\{ 0, \cdots, \ell_{\text{max}} \right\}\times\left\{ 0, \cdots, \ell_{\text{max}} \right\}$.
			Let $n_c\in\mathbb N$ denote the number of chromosomes in a cell, the configuration of a cell, given by the configurations of each of these chromosomes, is therefore an element of the set \[
			E_0:=\left(\left\{ 0, \cdots, \ell_{\text{max}} \right\}\times \left\{ 0, \cdots, \ell_{\text{max}} \right\}\right)^{n_c}.
			\]

				We choose a fixed quantity $\ell_{\text{min}}$, which denotes the minimum length a telomere can have: when the minimum length of the $3'$ ends of the telomeres within a cell falls below this value $\ell_{\text{min}} $, we will assume that the cell becomes senescent, that is, it will no longer be able to divide, although it may remain metabolically active in the body~\cite{CampisiJudith2007Cswb, diFagagnaFabriziodAdda2008Loab}

			Given a configuration $i=(i_{c,1},i_{c,2})_{c\in\{1,\ldots,n_c\}}\in E_0$, we define 
			\[ 
			\min i:=\min \{i_{1,1},i_{1,2}\ldots,i_{n_c,1}, i_{n_c,2}\},
			\]
			such that a cell with configuration $i$ is senescent if $\min i < \ell_{\text{min}}$, in which case it does not appear in the population count. Thus, the telomere length configuration of a non-senescent cell present in the population is an element of 
			\[ 
			E:=\left(\left\{ \ell_{\text{min}}, \cdots, \ell_{\text{max}} \right\}\times \left\{ \ell_{\text{min}}, \cdots, \ell_{\text{max}} \right\}\right)^{n_c}. 
			\]
			Since there are $L$ possible compartments, the type of a cell is an element of
			\[
			\mathbf X:=\{1,\ldots,L\}\times E,
			\]
			and, for $x=(k,i)\in\mathbf X$, we write $\min x:=\min i$.

				For $x=(k,i)\in \mathbf X$, we denote by $N_t^x$ the number of cells with type $x$, that is with  telomere length configuration $i \in E$ in compartment $k \in \llbracket 1; L \rrbracket$, at time $t \geqslant 0$. The overall configuration of the non-senescent cell population is described by 
				the empirical measure $\mu_t$ on $\mathbf X$ of the individual configurations of the cells in each compartment, defined as
				\begin{align}
					\label{eq:mut}
					\forall t \geqslant 0, \mu_t=\sum_{x\in\mathbf X} N_t^{x} \delta_{x}.%= \sum_{k=1}^L \sum_{i \in E}\mathds 1_{\min i\geqslant\ell_{\text{min}}} N_t^{k,i} \delta_{(k,i)}.
				\end{align} 
				The dynamics of the cell population are described by the following branching process, depending on the compartment of each cell. Note any daughter cell $x\in\mathbf X$ such that $\min x < \ell_{\text{min}}$ enters a senescent state and is removed from the system.
				
				\medskip\noindent \textbf{Compartment 1.} Each cell in compartment $1$ divides at an exponential rate $\displaystyle \tau_1>0$, independently of the rest of the system.
				Upon division, a cell with configuration $i \in E$ produces two cells with configurations $C_{i,1}^{+}$ and $C_{i,2}^{+}$ that possess telomerase (the distribution of these configurations is described below). Furthermore, with probability $\alpha_{R,1}\in(0,1)$, both daughter cells remain in compartment $1$, and with probability  $\alpha_{D,1}:=1-\alpha_{R,1}$, one of the two daughter cells remains in compartment $1$, and the other moves to compartment $2$. Thus, the part of the evolution of $\mu_t$ associated with compartment $1$ can be described as follows: for all  $i\in E$,
				\[
				\mu_t \rightarrow \left\{\begin{matrix}
					\mu_t - \delta_{(1,i)}+\delta_{(1,C_{i,1}^+)}\mathds 1_{C_{i,1}^+\geqslant\ell_{\text{min}}}+\delta_{(1,C_{i,2}^+)}\mathds 1_{C_{i,2}^+\geqslant\ell_{\text{min}}} & \text{at rate }   \alpha_{R,1}\tau_1 N_t^{(1,i)} &\\
					\mu_t - \delta_{(1,i)}+\delta_{(1,C_{i,1}^+)}\mathds 1_{C_{i,1}^+\geqslant\ell_{\text{min}}}+\delta_{(2,C_{i,2}^{+})} \mathds 1_{C_{i,2}^+\geqslant\ell_{\text{min}}}& \text{at rate } \alpha_{D,1}\tau_1 N_t^{(1,i)} 
				\end{matrix}\right.
				\]
				\medskip \noindent \textbf{Compartment $ 2 \leq k \leq L-1$.} Each cell in compartment $k$ divides at an exponential rate $\displaystyle \tau_k>0$ (which may depend on the compartment), independently of the rest of the system.
				Upon division, a cell with configuration $i \in E$ gives rise to two cells with configurations $C_{i,1}^{-}$ and $C_{i,2}^{-}$ that lack telomerase (the distribution of these configurations is described below). 
				Furthermore, with probability $\alpha_{R,k}\in(0,1)$, both daughter cells remain in compartment $k$, and with probability  $\alpha_{D,k}:=1-\alpha_{R,k}$, one of the two daughter
				cells remains in compartment $k$, and the other moves to compartment $k+1$.
				
				Thus, we can describe the part of the evolution of $\mu_t$ related to compartment $k$ as
				follows: for all $i\in E$,
				\[
				\mu_t \rightarrow 
				\begin{cases}
					\mu_t - \delta_{(k,i)}+\delta_{(k,C_{i,1}^-)}\mathds 1_{C_{i,1}^-\geqslant\ell_{\text{min}}}+\delta_{(k,C_{i,2}^-)}\mathds 1_{C_{i,2}^-\geqslant\ell_{\text{min}}} & \text{at rate }   \alpha_{R,k}\tau_k N_t^{(k,i)} \\
					\mu_t - \delta_{(k,i)}+\delta_{(k,C_{i,1}^-)}\mathds 1_{C_{i,1}^-\geqslant\ell_{\text{min}}}+\delta_{(k+1,C_{i,2}^-)}\mathds 1_{C_{i,2}^-\geqslant\ell_{\text{min}}} & \text{at rate }   \alpha_{D,k}\tau_k N_t^{(k,i)}  \\
				\end{cases}
				\] 
				\medskip \noindent \textbf{Compartment $L$.} The cells in compartment $L$ do not divide, but die at a rate $d_L > 0$.
								Thus, we can describe the part of the evolution of $\mu_t$ associated with compartment $k$ as follows: for all $i \in E$,
				\[
				\mu_t \rightarrow 		
					\mu_t - \delta_{(L,i)} \text{ at rate }   d_L N_t^{L,i}.
				\]
				
				\medskip
				Let us now describe the distributions of the configurations $(C^-_{i,1}, C^-_{i,2})$ (resp. $(C^+_{i,1},C^+_{i,2})$) of the telomere lengths of the two daughter cells of a mother cell with configuration $i\in E$, in the absence of telomerase (resp. in the presence of telomerase). This dynamics is illustrated in Figure~\ref{figure11}, adapted from Figure~\ref{figure9}  and supplemented by a final step corresponding to telomerase.

				%htbp
				\begin{figure}
					\centering
					\includegraphics[width=0.8\textwidth]{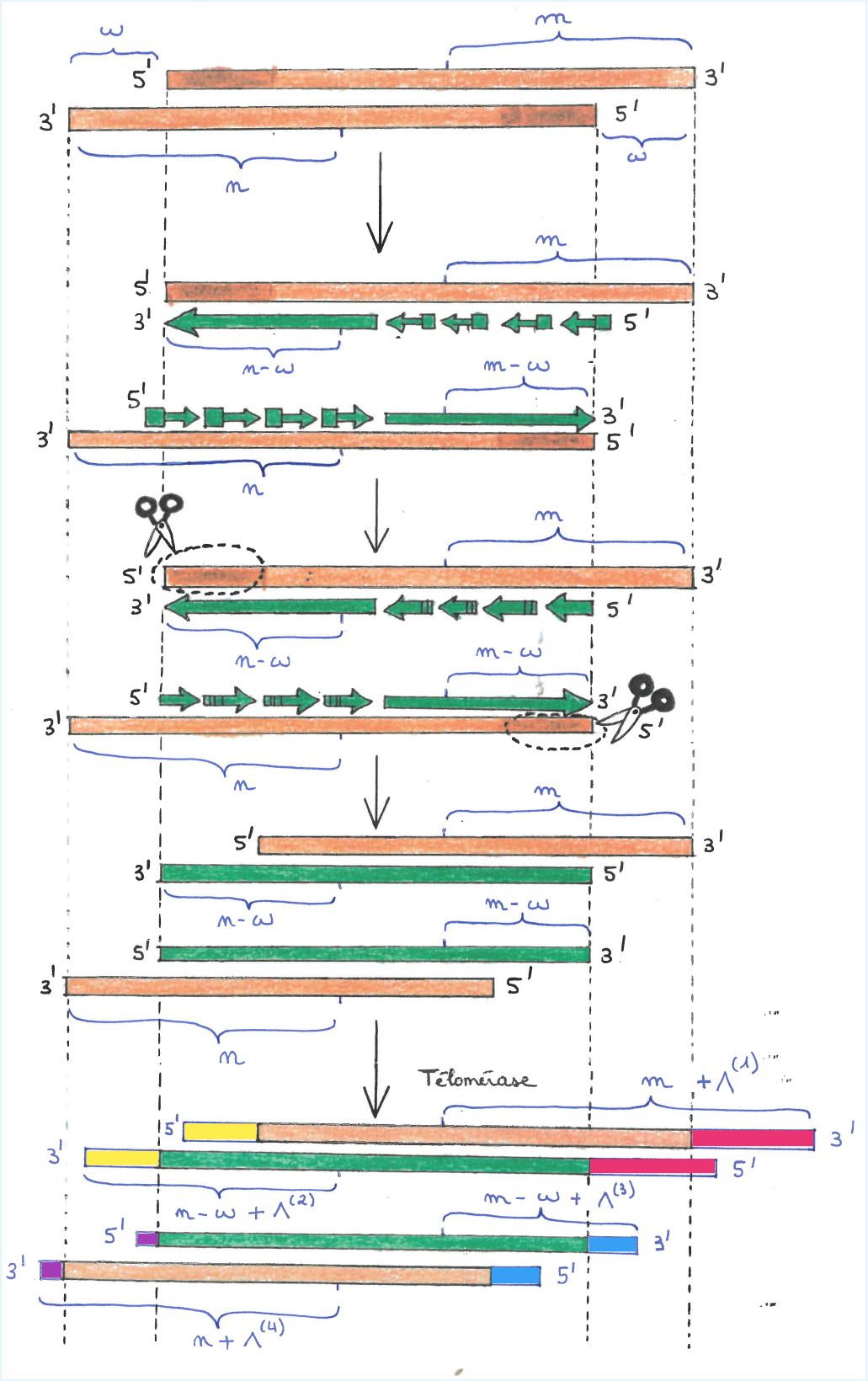}
					\caption{Length of telomeres during DNA replication before and after telomerase activity. Note that it may be the case that telomerase do not elongate all ends: for instance we may have $\Lambda^{(3)}=0$ in the above figure.
						} 
					\label{figure11}
				\end{figure}

				\medskip\noindent \textbf{Configurations $(C^-_{i,1},C^-_{i,2})$.}
				Configurations of this type correspond to the stage without telomerase.
				During DNA replication without telomerase, a chromosome with $3'$ telomere lengths $(m,n)$ is replicated into two chromosomes with respective $3'$ telomere lengths $(m,n-\omega)$ and $(m-\omega,n)$, where $\omega$ is the length of the overhang. The two chromosomes thus obtained after replication are distributed between the two daughter cells with uniform randomness, independently of the other chromosomes.

				Thus, formally, the two configurations $C^-_{i,1}$ and $C^-_{i,2}$ of the two daughter cells obtained from a parent configuration $i=(i_{c,1},i_{c,2})_{c\in \{1,\ldots,n_c\}}$ are given by
				\begin{align*}
					C^-_{i,1}&=\left(i_{c,1}-\omega B_c,i_{c,2}-\omega (1-B_c)\right)_{c\in \{1,\ldots,n_c\}}\\
					C^-_{i,2}&=\left(i_{c,1}-\omega (1-B_c),i_{c,2}-\omega B_c\right)_{c\in \{1,\ldots,n_c\}},
				\end{align*}
				where $B_c\in\{0,1\}$, $c\in\{1,\ldots,n_c\}$, are independent Bernoulli random variables with parameter $\displaystyle\frac{1}{2}$, and $c$ is the chromosome number in the cell.
				Thus, if $B_c=0$ (or $B_c=1$), a parent chromosome with configuration $i=(i_{c,1},i_{c,2})_{c\in \{1,\ldots,n_c\}}$ will produce two daughter chromosomes with configurations $(i_{c,1},i_{c,2} -\omega)_{c\in \{1,\ldots,n_c\}}$ for the first and $(i_{c,1}-\omega,i_{c,2})_{c\in \{1,\ldots,n_c\}}$ for the second (respectively, $(i_{c,1} -\omega,i_{c,2})_{c\in \{1,\ldots,n_c\}}$ for the first and $(i_{c,1},i_{c,2}-\omega)_{c\in \{1,\ldots,n_c\}}$ for the second).
				
				\medskip\noindent \textbf{Configurations $(C^+_{i,1},C^+_{i,2})$.} Configurations of this type correspond to the telomerase step, illustrated in the last step of Figure~\ref{figure11}.
				During DNA replication with telomerase, a chromosome with $3'$ telomere lengths $(m,n)$ is replicated into two chromosomes, each with independently elongated $3'$ telomeres, with respective probabilities depending on their respective lengths $m$ and $n$. The two chromosomes thus obtained after replication are distributed between the two daughter cells with uniform randomness, independently of the other chromosomes. Formally, assuming that a telomere of length $\ell$ is extended by the action of telomerase with probability $q(\ell)$, where $q:\mathbb N\to[0,1]$, by a random length following a distribution $\pi$ on $\mathbb N$, we obtain 
				\begin{align*}
					C^+_{i,1}&=\left(\ell_{\text{max}}\wedge (i_{c,1}-\omega B_c+\chi^{(1)}_c\Lambda^{(1)}_c),\ell_{\text{max}}\wedge(i_{c,2} -\omega (1-B_c)+\chi^{(2)}_c\Lambda^{(2)}_c)\right)_{c\in \{1,\ldots,n_c\}}\\
					C^+_{i,2}&=\left(\ell_{\text{max}}\wedge(i_{c,1} -\omega (1-B_c)+\chi^{(3)}_c\Lambda^{(3)}_c),\ell_{\text{max}}\wedge(i_{c,2}-\omega B_c+\chi^{(4)}_c\Lambda^{(4)}_c)\right)_{c\in \{1,\ldots,n_c\}},
				\end{align*}
				where $B_c\in\{0,1\}$, $c\in\{1,\ldots,n_c\}$, are independent random variables following a Bernoulli distribution with parameter $\displaystyle \frac{1}{2}$, $\chi^{(1)}_c$ follows a Bernoulli distribution with parameter $q(i_{c,1})$, $\chi^{(2)}_c$ follows a Bernoulli distribution with parameter $q(i_{c,2})$, $\chi^{(3)}_c$ follows a Bernoulli distribution with parameter $q(i_{c,1})$, $\chi^{(4)}_c$ follows a Bernoulli distribution with parameter $q(i_{c,2})$, and $\Lambda^{(1)}$, $\Lambda^{(2)}$, $\Lambda^{(3)}$, $\Lambda^{(4)}$ all follow the same distribution $\pi$, and all these variables are independent of one another.

				\section{Large-Population Model}

				In this section, we examine the behavior of the model in a large population, that is, when the number of cells in the system is very large. In this situation, cellular events (division, specialization, senescence) are so numerous that the stochastic aspect is averaged out, and the probabilistic model can be approximated by a deterministic model. The objective of this section is to demonstrate this phenomenon, known as the law of large numbers, and to explicitly describe the limiting model in terms of the chosen biological parameters. In the following section, we will examine the error term—that is, the difference between the stochastic model and the limiting deterministic model.

				From a mathematical perspective, we introduce a scale parameter $K \in \mathbb N^*$ that measures the size of the initial cell population and weight each individual by $\displaystyle \frac{1}{K}$ in the representation~\eqref{eq:mut}. In the following theorem, we denote by $(\mu_t^K)_{t\geqslant0}$ the empirical measure of the configurations when the initial population size is of order $K$, so that $\displaystyle \mu_0^{K} \left( \mathbf X \right)\simeq K$. We consider the renormalized measure $\displaystyle \frac{\mu_t^{K}}{K}$ and describe its asymptotic behaviour for large $K$. 
				To describe the (deterministic) dynamics of the limit, we introduce the following cells division rates. Let 
				 $y=(\ell,j)\in\mathbf X$, we define respectively the cells replication ($R$), differentiation ($D$) and death ($M$) rates of a cell with type $y$ (that is in compartment $\ell$ and configuration~$j$):
				\begin{align*}
					\theta_R^{y}=\alpha_{R,\ell}\tau_\ell \mathds 1_{\ell \leq L-1},\  \theta_D^{y}=\alpha_{D,\ell}\tau_\ell \mathds 1_{\ell \leq L-1},\  \theta_M^{y}=d_L \mathds 1_{\ell=L}.
				\end{align*}
				and, for all $x=(k,i)\in\mathbb X$,%k\in\{1,\ldots,L\}$ and $i\in E$,
				\begin{align*}
					&B_R^{x,y}= \left(\mathds 1_{k=\ell}\mathbb P(C_{j,1}^\pm=i)+\mathds 1_{k=\ell}\mathbb P(C_{j,2}^\pm=i)-\mathds {\mathds 1_{k=\ell}}1_{i=j}\right)\mathds 1_{\ell\leq L-1},\\
					&B_D^{x,y}=\left(\mathds 1_{k=\ell}\mathbb P(C_{j,1}^\pm=i)+\mathds 1_{k=\ell+1}\mathbb P(C_{j,2}^\pm=i)-\mathds {\mathds 1_{k=\ell}}1_{i=j}\right)\mathds 1_{\ell\leq L-1}\\
					&B_M^{x,y}=-\mathds 1_{k=\ell=L,i=j},
				\end{align*}
				where $\pm=+$ if $\ell=1$, and $\pm=-$ if $\ell\geqslant2$.

			These coefficients indicate the mean evolution of the number of cells of type $x$ when a cell of type of $y$ undergoes an event $R$, $D$ or $M$.

	Then we define the matrix 
	\begin{align}
	\label{definitiondeAronde}
	A:= \left(\sum_{e\in\{R,D,M\}}
		B_{e}^{x,y}  \theta_e^{y} \, 
		\right) _{x,y \in \mathbf X}.
\end{align}

In the following result, we denote by $\mathbb D([0,T])$ the space of càdlàg functions taking values in $\mathbb R^{\mathbf X}$ and equipped with the Skorokhod topology.

					\begin{theorem}
						\label{thm:LGN}
						Assume that $\left(\mathbb E\left(\frac1K\mu_0^K(\mathbf X)\right)\right)_{K\geq 1}$ is bounded and that there exists a finite measure $\nu_0$ on $\mathbf X$ such that
						\[
						\frac{\mu_0^{K}}{K} \xrightarrow[K \to +\infty]{\text{in law}} \nu_0,
						\]
						Then
						\[
						\left(\frac{\mu_t^{K}}{K}\right)_{t\in[0,T]} \xrightarrow[K \to +\infty]{\text{in law on } \mathbb D([0,T])}(\nu_t)_{t\in[0,T]},
						\]
						where $\nu$ is the unique solution on $\mathbb R_+^{\mathbf X}$ to the differential equation $\displaystyle \frac{\mathrm{d}\nu_t}{\mathrm{d}t}= A\nu_t$ with initial condition $\nu_0$.
						%{\color{blue} where $\nu$ is a probability measure on \textbf{$X$}.}
				\end{theorem}
				\begin{proof}

				The proof is, as is standard, based on a compactness-uniqueness argument in  $\mathbb D([0,T])$. In a first step, we decompose the process as a semi-martingale. In a second step, we prove first moment estimates for $L^1$ norm of $\mu^K_t$ and obtain the tightness result. In a third step, we identify the unique limit point as the unique solution to an ODE.

					\medskip\noindent

				\textit{Step~1. Semimartingale decomposition.} 
					In what follows, we write encode the random variables $C_{j,1}^\pm$ as functions of a uniformly random variable on $[0,1]$, which we abusively denote by $C_{j,1}^\pm(\zeta)$, $\zeta\in[0,1]$.
					In addition, for any event $e \in \lrac{R, D, M}$ and types $x,y\in\mathbf X$, we define $\beta_e^{x,y}$ as a random variable representing the evolution of the number of cells of type $x$ following an event of type $e$ occurring to a cell of type $y$. More precisely, we set for all $x=(k,i)\in\mathbf X$ and $y=(\ell,j)\in\mathbf X$ (all unspecified coefficients are equal to zero)
					\begin{align}
						&\beta_R^{x,y}(\zeta)= \left(\mathds 1_{k=\ell,\, C_{j,1}^\pm(\zeta)=i}+\mathds 1_{k=\ell,\,C_{j,2}^\pm(\zeta)=i}-{\mathds 1_{(k,i)=(\ell,j)}}\right)\mathds 1_{\ell\leq L-1},\notag\\
						&\beta_D^{x,y}(\zeta)=\left(\mathds 1_{k=\ell,\,C_{j,1}^\pm(\zeta)=i}+\mathds 1_{k=\ell+1,\, C_{j,2}^\pm(\zeta)=i}-{\mathds 1_{(k,i)=(\ell,j)}}\right)\mathds 1_{\ell\leq L-1}\label{eq:defbetas}\\
						&\beta_M^{(k,i),(\ell,j)}(\zeta)=-\mathds 1_{k=\ell=L,i=j},\notag
					\end{align}
					where $\pm = +$ if $\ell = 1$, and $\pm = -$ if $\ell \geqslant 2$. Note that 
					\[
					B_{e}^{x,y}=\int_0^1 \beta_{e}^{x,y}(\zeta)\,d\zeta = \mathbb{E}\left( \beta_{e}^{x,y} \right).
					\]

					As a consequence, we have, for all $1 \leq k \leq L$, we have
						\begin{multline}
							\label{eq:star21}
							\mu_t^K(x)=\mu_0^{K}(x)+\sum_{y\in \mathbf X}\sum_{e\in\{R,D,M\}}  \int_{[0,t]\times\mathbb{R}^+\times[0,1]} \beta_{e}^{x,y}(\zeta) \\
							\times \mathds{1}_{0 \leq u \leq \theta_e^{y}\mu^K_s(y)} \,\mathcal{N}_e^{y}\left(ds, du ,d\zeta \right),
						\end{multline}  where the $\mathcal{N}_e^{y}$, $y\in\mathbf X$, are independent Poisson point processes describing the random occurrence times of self-renewal (R), differentiation (D), and death (M) events for cells of type $y$.                 
						Consequently, we have
						\begin{align*}
							\mu^K_t=\mu^K_0+\sum_{y\in \mathbf X}\sum_{e\in\{R,D,M\}}  \int_0^t   B_{e}^{\cdot,y} \theta_e^{y}\mu^K_s(y) \, ds+\sum_{y\in\mathbf X}\sum_{e\in\{R,D,M\}} M_t^{K,e,y}
						\end{align*}
%						\begin{align*}
%								N_t^K=N_0^K+A_t^K+M_t^K
%						\end{align*} 
						where the $M^{K,e,y}=(M^{K,e,y,x})_{x\in\mathbf X}$, $e\in \{R,D,M\}, y\in\mathbf X$, are orthogonal $|\mathbf X|$-dimensional martingales such that
						\begin{align}
							\label{eq:bracket}
							\left\langle  \sum_{x\in\mathbf X} M^{K,e,y,x}\right\rangle_t= \int_0^t  C_{e}^{y}  \theta_e^{y}\mu^K_s(y) \, ds,
						\end{align}
						where
						\[
						C_{e}^{y}=\mathbb{E}\left[ \left( \sum_{x\in\mathbf X} \beta_{e}^{x,y} \right)^2 \right]\leq 4 .
						\]
						We therefore have
						\begin{align}
							\label{eq:star22}
							\mu_t^K=\mu_0^K+A_t^{K}+M_t^{K}
						\end{align} where
						\begin{align*}
							A_t^{K}=\int_0^t A \mu_s^K \,  \mathrm{d}s\quad \text{ and }\quad M_t^{K}= \sum_{y\in\mathbf X}\sum_{e\in\{R,D,M\}} M_t^{K,e,y},
						\end{align*}
						where $A$ is defined in~(\ref{definitiondeAronde}).

					\medskip\noindent\textit{Step~2. First moment estimates and tightness.} The aim of this step is to derive bounds on the first moments of $\mu^K$, $A^{K}$, and $M^{K}$, and then to prove the tightness in $D([0,T])$ of the family $\displaystyle \lr{\frac1K\mu^K_t}_{t\in[0,T]}$, $K\geqslant1$.
						We first show that
						\begin{align}
							\label{eq:Step2eq1}
							\mathbb{E}\left(  \sup_{t \in \left[ 0,T \right]} \mu^K_t(\mathbf X)\right) \leq \mathbb{E}\left(\mu^K_0(\mathbf X)\right) e^{2C_\theta T},
						\end{align}
						where 
						$
						C_{\theta}:=\sup_{y\in\mathbf X} \sum_{e\in\{R,D,M\}}\theta_e^y.
						%\sup_n\frac{\|f_R(n)\|}{\|n\|}+\frac{\|f_D(n)\|}{\|n\|}+\frac{\|f_M(n)\|}{\|n\|}.
						$
						We then deduce that, for all $\delta>0$ and $S,S'$ such that $S \leq S' \leq (S+\delta)\wedge T$ and $\left| S-S' \right|\leq \delta$,
						\begin{align}
							\label{eq:Step2eq2}
							\mathbb{E} \left\| A_{S}^{K}-A_{S'}^{K} \right\|\leq  \delta C_{A}\mathbb{E}\left(\mu^K_0(\mathbf X)\right) e^{2C_\theta T},
						\end{align}
						$\|\cdot\|$ denotes the $L^1$ distance, and 
						where $ \displaystyle  C_{A}:=\sup_{y\in\mathbf X}\sum_{x\in\mathbf X} |A_{x,y}|$.
						We then prove that
						\begin{align}
							\label{eq:Step2eq3}
							\mathbb{E}\left\| \sum_{e,y} M_S^{K,e,y}- \sum_{e,y} M_{S'}^{K,e,y} \right\|\leq  2 \sqrt{C_\theta \delta   \mathbb{E}\left(\mu_0^K(\mathbf X)\right)} e^{C_\theta T}|\mathbf X|.
						\end{align}
						Once proven, these three properties, the equality~\eqref{eq:star22} and the fact that $\left(\mathbb E(\frac1K\mu_0^K(\mathbf X))\right)_{K\geq 1}$ is bounded, imply that 
						\[
						\sup_{K\geqslant1}\sup_{\substack{S,S' \in \tau \\ S \leq S' \leq T+\delta}}  \mathbb{E}\left\|\frac{\mu_S^K}{K}-\frac{\mu_{S'}^K}{K}\right\|\xrightarrow[\delta \to 0]{}0
						\]
						According to~\eqref{eq:Step2eq1}, the sequence of laws of $(\frac1K\mu_t^K)_{t\geq 0,K\geq 1}$ is tight,
						and consequently, according to the Aldous criterion (see e.g.~\cite[p.21]{bansaye2017stochasticmodelsstructuredpopulations})
						 that the family $\lr{\frac1K\mu^K_t}_{t\in[0,T]}$, $K\geqslant1$, is tight in $\mathbb D([0,T])$.
						
						Let us first prove~\eqref{eq:Step2eq1}.    
						Let $0 \leq t \leq T$. By~\eqref{eq:star21}, we have, for all $x\in \mathbf X$,
						\begin{align*}
							\mu_t^{K}(\mathbf X)&\leq \mu^K_0(\mathbf X)+\sum_{e,x,y}  \int_0^t \int_{\mathbb{R}^+} \int_0^1 \left| \beta_{e}^{x,y}(\zeta) \right| \mathds{1}_{0 \leq u \leq \theta_e^{y}\mu_s^K(y)} \mathcal{N}_e^{y}\left(ds, du ,d\zeta \right)\\
							&\leq  \mu^K_0(\mathbf X)+2\sum_{e,y}  \int_0^t \int_{\mathbb{R}^+} \int_0^1\mathds{1}_{0 \leq u \leq \theta_e^{y}\mu_s^K(y)} \mathcal{N}_e^{y}\left(ds, du ,d\zeta \right),
						\end{align*}
						where we used the  fact that $\displaystyle \lrva{\sum_{x}\beta_{e}^{x,y}(\zeta)}\leq2$.
						Consequently,
						\begin{align*}
							\mathbb{E}\left(  \sup_{t \in \left[ 0,T \right]}\mu^K_t(\mathbf X) \right)
							&\leq \mathbb{E}\left( \mu_0^{K}(\mathbf X)\right) + 2\sum_{e,y} \int_0^T 
							\theta_e^y\mathbb{E}\left(\mu^K_s(y))\right) \mathrm{d}s\\
							&\leq  \mathbb{E}\left( \mu_0^{K}(\mathbf X)\right) + 2C_\theta\int_0^T  \mathbb{E}\left(  \mu^K_s(\mathbf X)  \right) \mathrm{d}s,\\
							& \leq \mathbb{E}\left( \mu_0^{K}(\mathbf X)\right)+2C_\theta\int_0^T 
							\mathbb{E}\left( \sup_{t \in \left[ 0,s \right]}  \mu_t^{K}(\mathbf X) \right)
							\mathrm{d}s.
						\end{align*}
						This, together with Gronwall’s lemma, allows us to deduce \eqref{eq:Step2eq1}.
						To prove ~\eqref{eq:Step2eq2}, let us now observe that
						\begin{align*}
							\left\|A_{S'}^{K}-A_{S}^{K}\right\|&=\int_S^{S'}\left\| A\mu_u^K\right\| \mathrm{d}u
							\leq \int_S^{S'}C_A\mu_u^K(\mathbf X) \mathrm{d}u
							\leq{\underbrace{\left| S'-S \right|}_{\leq \delta}} C_A \sup_{u \leq T}\mu_u^K(\mathbf X).
						\end{align*}
						Therefore, taking expectations and using \eqref{eq:Step2eq1}, we obtain~\eqref{eq:Step2eq2}.

						Let us conclude this step by proving \eqref{eq:Step2eq3}. We deduce from~\eqref{eq:bracket} that, for all $x\in\mathbf X$,
						\begin{align}
							\sum_{e,y}  \left| \left<M^{K,e,y,x}\right>_{S'}-\left<M^{K,e,y,x}\right>_{S}  \right|
							&
							\leq 4 C_\theta \int_S^{S'} 
							\mu_s(\mathbf X) \, \mathrm{d}u \\
							&\leq  4 C_\theta \delta  \sup_{t \in \left[ 0, T \right]}  \mu_t^K(\mathbf X).
						\end{align}
						Taking the expected value and using~\eqref{eq:Step2eq1}, we therefore obtain
						\begin{align}
							\sum_{e,y} \mathbb E \left| \left<M^{K,e,y,x}\right>_{S'}-\left< M^{K,e,y,x}\right>_{S}  \right|&\leq  
							 4 C_\theta \delta   \mathbb{E}\left(\mu_0^K(\mathbf X)\right) e^{2C_\theta T},
							\label{cstetendverso}
						\end{align}
						which, using the orthogonality of the martingales involved, allows us to deduce \eqref{eq:Step2eq3} and conclude this step.

					\medskip\noindent\textit{Step 3. Identification of the limit.} Since the tightness of the family of the distributions of $(\frac1K\mu^K_t)_{t\in[0,T]}$, $K\geqslant1$, was established in the previous step, we now prove that every limit point is the solution to an equation that has a unique solution.

						Up to extraction and by Prokhorov’s theorem, there exists a random variable $X^{\infty}$ taking values in  $\mathbb{D}\left( \left[ 0,T \right] \right)$ such that $ \frac1K \mu^K$ converges in distribution to $X^{\infty}$. Let us first show that $X^{\infty}$ is continuous almost surely. To do this, we introduce the following function, which is continuous on $\mathbb D([0,T])$ and which vanishes if and only if the trajectory is continuous:
						\[
						\Delta :
						\begin{array}{cccc}
							\mathbb D & \to & \mathbb{R}_+ \\
							\left(x(t)\right)_{t \in \left[0,T  \right]} & \mapsto & \displaystyle \sup_{t \in \left[0,T  \right]} \left\| x(t)-x(t_-) \right\| \wedge  1 
						\end{array} 
						\] 
						Since the function $\Delta$ is continuous and bounded on $\mathbb D$, we obtain
						\begin{equation*}
							\mathbb E \left[ \Delta\left( \frac{\mu^K}{K} \right) \right]
							\underset{K\to +\infty}{\longrightarrow}\mathbb E \left[  \Delta \left(X^{\infty}\right)  \right]
						\end{equation*}
						From the definitions of $\Delta$ and of $\mu^K$, we have, for all $K\geqslant1$,  $\displaystyle\Delta\left( \frac{N^K}{K} \right) \leq \frac{2}{K}$, so
						$ \mathbb{E}\left( \Delta\left( X^{\infty} \right) \right)=0$. We conclude that $\Delta\left( X^{\infty} \right)=0$ almost surely, and consequently, $X^{\infty}$ is almost surely continuous.

						To determine the limit, we define, for all $t\in[0,T]$:
						\[
						\psi_t:
						\begin{array}{ccl}
							\mathbb D([0,T]) & \to & \mathbb{R}^{\mathbf X} \\
							\left(x(s)\right)_{s \in \left[0,T  \right]} & \mapsto &\displaystyle x(t)-x(0)- \int_0^t Ax(s)\mathrm{d}s  \\
						\end{array}.
						\]
						Since the random trajectory $X^\infty$ is almost surely continuous, it follows that  $\psi_t$ is almost surely continuous at $X^\infty$. Consequently, 						
						\begin{align*}
							\mathbb{E}\left( \left\|\psi_t\left( X^{\infty}\right)\right\|  \wedge 1  \right)&=\lim_{K\to\infty} \mathbb{E}\left(  \left\|\psi_t\left( \frac{\mu^K}{K}\right) \right\|  \wedge 1  \right)\\
							&= \lim_{K\to\infty}  \mathbb{E}\left( \left\| \frac1K M^{K}_t \right\|  \wedge 1  \right)
							=0
						\end{align*}
						according to ~\eqref{eq:Step2eq3} with $S=0, S'=T, \delta=T$.
						It follows that $\psi_t\left( X^{\infty}\right)=0$ almost surely, that is, almost surely, $X^\infty$ satisfies $X^\infty_0=\nu_0$ and is a solution to
						\[
						\frac{\mathrm{d}X_t^{\infty}}{\mathrm{d}t}= AX^\infty_t.
						\]
						Since this equation admits a unique solution, it is therefore the limit in distribution of $(X^K)_{K\geqslant1}$ in $\mathbb D([0,T])$.
				
			\end{proof}		
					
						\section{Fluctuations and  Central Limit Theorem}
						
						We focus in this section on the fluctuations of $\left(\frac{\mu_t^K}{K}\right)_{t\geq 0}$ around its limit $(\nu_t)_{t\geq 0}$, characterized in the previous section as the unique solution to the equation $\frac{\mathrm d\nu_t}{\mathrm dt}= A\nu_t$ with initial condition $\nu_0$. The following theorem shows that these fluctuation can be described by a Gaussian process. In order to define the limit, let, for all $t\geq 0$, $\sigma_t\in \mathbb R_+^{\mathbf X\times\mathbf X}$ be a matrix such that 
						\[
						\sigma_t^{\top}\sigma_t=\sum_{e,y} \theta_e^{y}\nu_t(y) \bar C_e^{y},
						\]
						where $\bar C_e^{y}\in \mathbb R_+^{\mathbf X\times\mathbf X}$ is the symmetric matrix defined by
						\[
						\left[\bar C_{e}^y\right]_{x,x'}:=\mathbb E\left(\beta_e^{x,y}\beta_e^{x',y}\right),
						\]
						with $\beta_e^{x,y}$ defined in~\eqref{eq:defbetas}.

	We recall that $\|\cdot\|$ denotes the $L^1$ norm.
	\begin{theorem}
		\label{thm:fluctuations}
		Take the assumptions of Theorem~\ref{thm:LGN}.
		For all $K\geq 1$ and $t\geq 0$, let 
		\[
		U_t^K:=\sqrt{K}\left[ \frac{\mu_t^K}{K} - \nu_t\right].
		\]
		We assume that $\left(\mathbb E\left(\left\|U_0^K\right\|^2\right)\right)_{K\geq 1}$ is bounded and that there exists a random variable $U_0^\infty$ such that
		\begin{align*}
			U^K_0\xrightarrow[K\to+\infty]{\text{loi}} U_0^\infty.%\text{ et }\sup_{K\geq 1} \mathbb E\left[\|U^K_0\|^2\right]<\infty.
		\end{align*}
		Then, for all $T>0$,
		\begin{align*}
			U^K& \xrightarrow[K \to +\infty]{\text{in law in }\mathbb D([0,T]) } U,
		\end{align*}
		 where 
		\[
		U_t=e^{tA}U_0+\int_0^t e^{(t-s)A}\sigma_s\,\mathrm dW_s,
		\]
		with $W$ a standard $|X|$-dimensional Brownian motion.
	\end{theorem}

	\begin{proof}
		
	The proof is decomposed into several steps. We prove uniform bounds on the first and second moments of $\sup_{u\in[0,t]}\left\|  U_u^K \right\|$ in Steps~1 and~2 respectively. In Step~3, we prove the tightness result. In Steps~4 and~5, we identify the infinitesimal generator associated to any limiting process and show that it is solution to the associated martingale problem. We conclude in Step~6 by identifying the unique limit point as the unique solution to the martingale problem of Step~5.

		\medskip\noindent \textit{Step 1. Bound on the first moment of $\sup_{u\in[0,t]}\left\|  U_u^K \right\|$.}
		We show that, for all $T>0$, there exists a constant $C_T$ such that
		\begin{align}
			\label{eq:step1U}
			\mathbb E\left(\sup_{u\in[0,t]}\left\|  U_u^K \right\| \right)\leq C_{T},\quad \forall t\in[0,T].
		\end{align}
		According to~\eqref{eq:star22}, we have, for all $t\in[0,T]$,
		\begin{align}
			\label{eq:decompU}
			U_t^{K}=
			U_0^K+\widetilde{A}_t^{K}+ \frac{M_t^{K}}{\sqrt{K}}+\Delta_t^{K}-\Delta_0^{K},
		\end{align}
		where $M^K$ is defined in the proof of Theorem~\ref{thm:LGN} and, setting $\nu^{K}_t:=\mathbb{E}(\frac1K \mu^K_t)$,
		\[
		\widetilde{A}_t^{K}=\sqrt{K} \left[ \frac{A_t^{K}}{K} -(\nu^{K}_t-\nu^{K}_0) \right]
		\text{ and }\Delta_t^K:=\sqrt{K}(\nu_t^{K}-\nu_t).
		\]
		We have
		\begin{align*}
			\widetilde{A}_t^{K}&=\sqrt{K} \left[\frac1K\int_0^{t} A \mu_u^K\,\mathrm{d}u -\int_0^{t} A\nu_u^K \mathrm{d}u \right]=\int_0^{t}  A U_u^K \, \mathrm{d}u,
		\end{align*}
		since $\nu^K$ is solution to  $\frac{\mathrm d\nu^K_t}{\mathrm dt}=A\nu^K_t$. 
		Hence, we obtain (recall that $C_A$ is defined in the proof of Theorem~\ref{thm:LGN})
		\begin{align}
			\label{eq:tildeA1}
			\|\widetilde{A}_t^{K}\|\leq C_{A} \, \int_0^{t} \|U_u^K\|\,\mathrm du.
		\end{align}
		
		We focus now on $M_t^K$. We have, using the same notations as in the proof of Theorem~\ref{thm:LGN},
		\[
		\left\|M_t^K\right\|=\sum_x \left|\sum_{e,y} M_t^{K,e,y,x}\right|,
		\]
		with, for any fixed $x$, the martingales $M^{K,e,y,x}$, $e\in\{R,D,M\}, y\in\mathbf X$, orthogonal to each others and such that
		\[
		\left\langle M^{K,e,y,x} \right\rangle_t=\int_0^t C_{e}^{x,y}\,\theta_e^y\mu_s^K(y)\,\mathrm ds,
		\]
		where
		\[
		C_e^{x,y}=\mathbb E\left[\left(\beta_e^{x,y}\right)^2\right]\leq 2.
		\]
		We thus have, for all $x\in\mathbf X$,
		\begin{align}
			\label{eq:usefulCompM}
		\left\langle \sum_{e,y} M^{K,e,y,x}\right\rangle_t &=  \sum_{e,y} \left\langle M^{K,e,y,x}\right\rangle_t\leq 2C_\theta\int_0^t \mu_s^K(\mathbf X)\,\mathrm ds,
		\end{align}
		and hence, using Doobs maximal inequality and~\eqref{eq:Step2eq3}, we get
		\begin{align*}
			\mathbb E\left(\sup_{t\in [0,T]} \left| \sum_{e,y} M_t^{K,e,y,x}\right|\right)\leq \sqrt{2C_\theta \int_0^T \mathbb E\left(\mu_s^K(\mathbf X)\right)\,\mathrm ds} \leq 2\, \sqrt{2C_\theta T \mathbb E(\mu_0^{K}(\mathbf X))}\,e^{C_\theta T}.
		\end{align*}
	Since $\frac{\mathbb{E}\left(\left\| \mu^K_0 \right\| \right)}{K}$ is uniformly bounded in $K\geq 1$ by assumption, we deduce from the fact that $\mathbf X$ is finite that there exists a constant $C_{T,M}$ such that
		\begin{align}
			\label{eq:borneMtsurK}
			\mathbb{E}\left(  \sup_{t \in \left[0,T  \right]} \left\| \frac{M_t^{K}}{\sqrt{K}}\right\| \right)\leq C_{T,M}.
		\end{align}
		
		In addition, since $\nu^K_t$ and $\nu_t$ are solution to the equation $\frac{\mathrm dx_t}{\mathrm dt}=Ax_t$, we deduce that
		\begin{align*}
			\|\Delta^{K}_t\|\leq \sqrt{K}\|\nu_0^K-\nu_0\|\,e^{C_{A} t}\leq \mathbb E\|U_0^K\|\,e^{C_{A} t},
		\end{align*}
		and hence, since  $\mathbb E\|U_0^K\|$ is assumed to be uniformly bounded in $K$, there exists $C_{T,\Delta}\geq 0$ such that
		\begin{align}
			\label{eq:UborneDelta}
			\sup_{t\in[0,T]}\|\Delta^{K}_t-\Delta^{K}_0\| \leq C_{T,\Delta}.
		\end{align}
		
		Using the equality~\eqref{eq:decompU} and the inequalities~\eqref{eq:tildeA1},~\eqref{eq:borneMtsurK} and~\eqref{eq:UborneDelta}, we deduce that, for all $t\in[0,T]$,
		\begin{multline*}
			\mathbb E\left(\sup_{u\in[0,t]}\left\|  U_u^K \right\|\right)\leq \mathbb E \left\|  U_0^K \right\|+C_{A}\int_0^t \mathbb E\left(\sup_{u \in \left[ 0,v \right]} \|U_u^K\| \right)  \mathrm dv
			+C_{T,M}+C_{T,\Delta}.
		\end{multline*}
		We deduce from Grönwall Lemma that~\eqref{eq:step1U} holds true.
		
	\medskip\noindent\textit{Step 2. Uniform bound on the second moment of  $\sup_{u\in[0,t]}\left\|  U_u^K \right\|$.}
		We show that there exists a constant $C'_{T}>0$ such that
		\begin{align}
			\label{eq:Ustep2}
			\mathbb E\left(\left(\sup_{u\in[0,t]}\left\|  U_u^K \right\|\right)^2\right)\leq C'_{T},\quad \forall t\in[0,T].
		\end{align}

		Starting with equality~\eqref{eq:decompU} and using the inequalities~\eqref{eq:tildeA1} and~\eqref{eq:UborneDelta}, we obtain, similarly as in the previous step that, for all $t\in[0,T]$,
		\begin{multline}
			\label{eq:Ustep2borne1}
			\sqrt{\mathbb E\left(\sup_{u\in[0,t]}\left\|  U_u^K \right\|^2\right)}\leq \sqrt{\mathbb E\left(\left\|  U_0^K \right\|^2\right)}
			+C_{A}\sqrt{\mathbb E\left(\left(\int_0^t \sup_{u \in \left[ 0,v \right]} \|U_u^K\|\,\mathrm dv \right)^2\right)}\\
			+\sqrt{\mathbb E\left(\left\| \frac{M^{K}_T}{\sqrt K}\right\|^2\right)}
			+ C_{T,\Delta}.
		\end{multline}
		In addition, using~\eqref{eq:usefulCompM} and~\eqref{eq:Step2eq1}, we deduce as in the previous step that 
		\begin{align*}
			\sqrt{\mathbb E\left(\left\| \frac{M^{K}_T}{\sqrt K}\right\|^2\right)}&\leq \sum_{x\in\mathbf X}\sqrt{\mathbb E\left(\left(\frac{\sum_{e,y} M^{K,e,y,x}_T}{\sqrt K}\right)^2\right)}\\
			& \leq |\mathbf X|\,2C_\theta T e^{2C_\theta T}  \frac1K \mathbb E(\mu^K_0(\mathbf X))\\
			&\leq C'_{T,M},
		\end{align*}
		where $C'_{T,M}$ is a constant which does not depend on $K$.

		We then deduce from  inequality~\eqref{eq:Ustep2borne1} that
		\begin{multline}
			\label{eq:Ustep2borne2}
			\sqrt{\mathbb E\left(\sup_{u\in[0,t]}\left\|  U_u^K \right\|^2\right)}\leq \sqrt{\mathbb E \left(\left\|  U_0^K \right\|^2\right)}
			+C_{A}\sqrt{\mathbb E\left(T\,\int_0^t \sup_{u \in \left[ 0,v \right]} \|U_u^K\|^2\,\mathrm dv \right)}\\
			+C'_{T,M}
			+ C_{T,\Delta}.
		\end{multline}
		Using Gronwall's Lemma, we deduce that~\eqref{eq:Ustep2} holds true.
	
		\medskip\noindent\textit{Step 3. Tightness of $(U^K)_K$.}
		
		Let $\delta>0$ and let $S,S'$ be stopping times such that $S \leq S' \leq (S+\delta)\wedge T$ almost surely. We have, similarly as in~\eqref{eq:tildeA1}, 
		\begin{align}
			\label{eq:UtildeA2}
			\mathbb E	\left\|\widetilde{A}_S^{K}-\widetilde{A}_{S'}^{K}\right\|&\leq C_{A} \mathbb E\left(\int_S^{S'} \|U^K_u\|\,\mathrm du\right)\leq  C_A\,\delta  \,   \mathbb E\left(\sup_{u \in \left[ 0,T \right]} \left\| U_u^K\right\|\right)
			\leq \delta C_A C_{T}.
		\end{align}
		In addition, similarly as in Step~1, we obtain
		\begin{align*}
			\mathbb{E} \left| \sum_{e,y} M^{K,e,y,x}_t- M^{K,e,y,x}_{S'}  \right| &\leq  \sqrt{2C_\theta \delta \mathbb E(\mu_0^{K}(\mathbf X))}\,e^{C_\theta T}.
		\end{align*}
		Hence there exists a constant $C_{T,M}''>0$ such that
		\begin{align*}
			\mathbb{E} \left(\left\| \frac1{\sqrt K} M^{K}_S- \frac1{\sqrt K} M^{K}_{S'}  \right\| \right)&\leq \sqrt{\delta}\, C''_{T,M}.
		\end{align*}
Finally, we also have
		\begin{align*}
			\frac1{\sqrt K}\left\| \Delta^{K}_S- \Delta^{K}_{S'}  \right\| &= \left\| (\nu_S^{K}-\nu_S)- (\nu_{S'}^{K}-\nu_{S'})\right\|
			\leq\left\|\nu^K_S-\nu_S\right\| \,(e^{C_{A}\delta}-1)\\
			&\leq  \left\| \Delta^{K}_S\right\| \,(e^{C_{A}\delta}-1)\leq C_{T,\Delta} \,(e^{C_{A}\delta}-1).
		\end{align*}
		Using the previous inequalities and the decomposition~\eqref{eq:decompU}, we deduce that
		\begin{align*}
			\mathbb{E}\left( \left\| U_{S}^K-U_{S'}^K \right\|\right)
			&\leq \delta C_A C_{T}+C''_{T,M}\sqrt{\delta}+C_{T,\Delta}\,(e^{C_{A}\delta}-1)
		\end{align*}
		and thus conclude that
		\begin{align*}
			\sup_{\substack{S,S' \in \tau,\ S \leq S' \leq T+\delta\\ K\geq 1} } \mathbb{E}\left( \left\| U_{S}^K-U_{S'}^K  \right\| \right) \underset{\delta \to 0}{\longrightarrow}0.
		\end{align*}
		In addition, according to Step~2, $\|U^K_s\|$, $K\geq1,s\in[0,t]$, is uniformly integrable and hence the sequence defined by $\left( U_s^K \right)_{s \in \left[ 0,T \right]}$, $K\geqslant1$, is tight in $\mathbb D([0,T])$ by Aldous tightness criterion (see e.g.~\cite[p.21]{bansaye2017stochasticmodelsstructuredpopulations})

		\medskip\noindent\textit{Step 4. Identification of the infinitesimal generators of the limit.}
		In what follows, we denote by $\mathscr L^K_t$ the infinitesimal generator of $U^K$ at time $t\geq 0$ and by $\mathscr H^K$ the infinitesimal generator of $\mu^K/K$ (which is time-homogeneous). For all  $f \in C_c^2\lr{\R^{\mathbf X},\R}$ (the set of twice continuously differentiable functions with compact support), we have
		\[
		f(U^K_t)=f\left(\sqrt K(\mu^K_t/K-\nu^K_t)\right)
		\]
		and hence, for all $t\geq 0$, $u\in \R^{\mathbf X}$ and setting $z=u/\sqrt{K}+\nu_t^K$,
		\begin{multline*}
			\mathscr{L}_t^Kf(u)=\mathscr{H}^Kf\lr{\sqrt{K}\lr{\cdot -\nu_t^K}}(z)\\
			+\left\langle \partial_t\lr{\sqrt{K}\lr{z-\nu_t^K}} ,\nabla f\lr{\sqrt{K}\lr{z-\nu_t^K}}\right\rangle.
		\end{multline*}
		We handle the first and then the second term on the right. We have for the first one
		\begin{align*}
			\mathscr{H}^Kf(z)=\sum_{e,y}\theta_e^{y} Kz_y\E\left[f\left(z+\sum_{x\in\mathbf X}\frac{\beta_e^{x,y}}{K}e_x\right)-f(z)\right],
		\end{align*}
		where $(e_x)_{x\in\mathbf X}$ denotes the canonical basis of $\mathbb R_+^{\mathbf X}$.
	As a consequence,
		\begin{align*}
			\mathscr{H}^Kf\big(\sqrt{K}&\lr{\cdot -\nu_t^K}\Big)(z)\\
			&=\sum_{e,y} \theta_e^{y} K z_y \E\Bigg[f\lr{\sqrt{K}\lr{z+\sum_{x}\frac{\beta_e^{x,y}}{K}e_x-\nu_t^K}}\\
			&\qquad\qquad\qquad\qquad\qquad\qquad-f\lr{\sqrt{K}\lr{z-\nu_t^K}}\Bigg]\\
			&=\sum_{e,y} \theta_e^{y}K z_y \E\lrcr{f\lr{u+\sum_{x}\frac{ \beta_e^{x,y}}{\sqrt K}e_x}-f\lr{u}}.
		\end{align*}
		We consider, for all $e,y$, the $\mathbb R^{\mathbf X}$-valued random vector $\bar \beta_e^y$, the vector $\bar B_e^y$ and the symmetric matrix $\bar C_e^y$ defined by
		\[
		\bar \beta_e^y:=\sum_{x\in\mathbf X}\beta_e^{x,y}e_x,\ \bar B_e^y:=\mathbb E(\bar\beta_e^y),\ \bar C_e^y:=\mathbb E\left(\bar\beta_e^{y}\,(\bar \beta_e^{y})^{\top}\right).
		\]
		We thus obtain, using an order 2 Taylor expansion en denoting by $H_f$ the Hessian matrix of $f$, 
		\begin{align*}
			\mathscr{H}^Kf\big(\sqrt{K}&\lr{\cdot -\nu_t^K}\Big)(z)\\
			&\underset{K \rightarrow \infty}{=}\sum_{e,y} \theta_e^{y}Kz_y \E\left(\frac1{\sqrt K}\nabla f(u)^{\top}\,\bar\beta_e^y+\frac1{2K}\, (\bar\beta_e^y)^{\top}\, H_f(u)\, \bar\beta_e^y+O(K^{-3/2})\right)\\
			&\underset{K \rightarrow \infty}{=}  \sum_{e,y} \theta_e^{y}z_y  \sqrt K\nabla f(u)^{\top}\,\bar B_e^y+ \sum_{e,y} \frac{\theta_e^y z_y}{2}\, \text{Tr}(\bar C_e^y\, H_f(u))+O(K^{-1/2}),
		\end{align*}
		where $O(K^{-3/2})$ and $O(K^{-1/2})$ are uniform over $t$ and $u$ in compact sets.
For the second term in the expression of  $\mathscr{L}_t^Kf(u)$, we use the fact that
		\begin{align*}
			\partial_t \nu_t^K=\sum_{e,y}   \bar B_{e}^{y}  \theta_e^{y}\nu_t^K(y)
		\end{align*}
		to deduce that
		\begin{align*}
			&\left\langle \partial_t\lr{\sqrt{K}\lr{z-\nu_t^K}} ,\nabla f\lr{\sqrt{K}\lr{z -\nu_t^K}}\right\rangle\\
			&\qquad\qquad\qquad\qquad\qquad\qquad=-\sqrt{K} \sum_{e,y}    \theta_e^{y}\nu_t^K(y)\,(\nabla f(u))^{\top}\bar B_{e}^{y} .
		\end{align*}
		We finally obtain, using also the fact that $z=u/\sqrt K+\nu^K_t$, with $\nu_t^K\to \nu_t$ as $K\to\infty$, that
		\begin{align*}
			\mathscr L^K_t f(u)
			&\underset{K \rightarrow \infty}{=}  \sum_{e,y} \theta_e^{y}u_y(\nabla f(u))^{\top}\bar B_{e}^{y} +\frac{\theta_e^{y}\nu_t^K(y)}{2}\, \text{Tr}(\bar C_e^y\, H_f(u))+O(K^{-1/2})
		\end{align*}
		%où $o\lr{\frac{1}{K}}$ est uniforme en $t$ et $u$ sur les compacts. 
		and we deduce that, for all compact subset $C\subset \mathbb R_+\times \mathbf X$,
		\[
		\lim\limits_{K \to +\infty} \, \sup_{(t,u)\in C} \lrva{ \mathscr{L}_t^{K}f\lr{u} -\mathscr{L}_t^{\infty}f\lr{u}}=0,
		\]
		where
		\begin{align*}
			\mathscr{L}_t^{\infty}f\lr{u} &:= \sum_{e,y} \, \theta_e^{y}u_y (\nabla f(u))^{\top}\bar B_{e}^{y}
				+\frac{\theta_e^{y}\nu_t(y)}{2}\, \text{Tr}(\bar C_e^y\, H_f(u))
		\end{align*}
		
		\medskip\noindent \textit{Step 5. Identification of the limit points infinitesimal generator.}
		For all $0\leq t\leq T$ and $f \in C_b^2\lr{\R^{L\times \lrva{E}},\R}$, we define on $\mathbb{D}\lr{\lrcr{0,T},\R^{L\times \lrva{E}}}$ the function
		\[
		\xi_t^f(u)= f\lr{u_t}-f\lr{u_0}-\int_0^t \mathscr{L}_s^{\infty}f\lr{u_s} \mathrm{d}s.
		\]
		According to Steps~2, the sequence $\displaystyle \lr{\xi_t\lr{U^K}}_{t\in[0,T],K\geq 1}$ is uniformly integrable.
		
		Let $Q$ be a limit law of the tight sequence $(U^K)$ and let $U^\infty$ be a random variable in $\mathbb D([0,T])$ with law $Q$.
		Our aim is to show that $\lr{\xi_t(U^\infty), t >0}$ is a martingale  $U^\infty$. This is true as soon as $\mathbb{E}\lr{H\lr{U^\infty}}=0$ for any function $H$ defined as follows
		\begin{align*}
			H\lr{u}=g_1\lr{u_{s_1}}\cdots g_k\lr{u_{s_k}}\lr{\xi_t\lr{u}-\xi_s\lr{u}}
		\end{align*}
		for $0 \leq s_1 \leq \cdots \leq s_k \leq s \leq t$ and $\displaystyle g_1, \cdots, g_k \in C_b(\mathbb R_+^{\mathbf X})$, which we will show now.

		For all $K\geq 1$, we set
		\begin{align*}
			H^K(u)=g_1\lr{u_{s_1}}\cdots g_k\lr{u_{s_k}}\lr{\xi^{K}_t\lr{u}-\xi^{K}_s\lr{u}}
		\end{align*}
		where
		\begin{align*}
			\xi^{K}_t(u)= f\lr{u_t}-f\lr{u_0}-\int_0^t \mathscr{L}_s^{K}f\lr{u_s} \mathrm{d}s .
		\end{align*}
		We have, for any $K_0$ large enough which will be chosen later,
		\begin{align}
			\E\lrcr{H\lr{U^\infty}}&= \E\lrcr{H\lr{U^\infty}-H^{K_0}\lr{U^\infty}} \label{eq:circle1}\\ %& \circled{1} \\
			&+   \E\lrcr{H^{K_0}\lr{U^\infty}-H^{K_0}\lr{U^{K}}}  \label{eq:circle2}\\ % & \circled{2} \\
			&+   \E\lrcr{H^{K_0}\lr{U^K}-H^{K}\lr{U^{K}}}  \label{eq:circle3}\\ %& \circled{3} \\
			&+ \E\lrcr{H^{K}\lr{U^K}}  \notag.%& \circled{4} \\
		\end{align}
		We first observe that $\E\lrcr{H^{K}\lr{U^K}} =0$ (since $\xi_t^{K,v}(U^K)$ is a uniformly integrable martingale) and we will show that each other term on the right hand side goes $0$ when $K\to+\infty$, which will then lead $\mathbb{E}\lr{H\lr{U^\infty}}=0$ and conclude this step. We have, for the first term~\eqref{eq:circle1},
		\begin{align*}
			\Big|\E\Big(H&\lr{U^\infty}-H^{K_0}\lr{U^\infty}\Big)\Big|\\
			&=\left|\E\left(g_1\lr{U^\infty_{s_1}}\cdots g_k\lr{U^\infty_{s_k}} \int_s^t \mathscr{L}_r^{K_0}f\lr{U^\infty_{r}}-\mathscr{L}_r^{\infty}f\lr{U^\infty_{r}}\mathrm{d}r \right)\right| \\
		&\leq \lrdb{g_1}_{\infty} \cdots \lrdb{g_k}_{\infty} \int_s^t \E\left({\lrva{
				\mathscr{L}_{r}^{K_0}f\lr{U_{r}}-\mathscr{L}_{r}^{\infty}f\lr{U_{r}} }}\right) \mathrm{d}r
	\end{align*}
	where
	\begin{align*}
		\E\Big(\lrva{
			\mathscr{L}_{r}^{K_0}f\lr{U^\infty_{r}}-\mathscr{L}_{r}^{\infty}f\lr{U^\infty_{r}} }\Big)&\leq \E\lrcr{\lrva{
				\mathscr{L}_{r}^{K_0}f\lr{U^\infty_{r}}-\mathscr{L}_{r}^{\infty}f\lr{U^\infty_{r}} } \mathds{1}_{\lrdb{U^\infty_{r}}\leq A}}\\ 
		&\qquad\qquad+\E\lrcr{\lrva{
				\mathscr{L}_{r}^{K_0}f\lr{U^\infty_{r}}-\mathscr{L}_{r}^{\infty}f\lr{U^\infty_{r}} } \mathds{1}_{\lrdb{U^\infty_{r}}> A}}\\
		&\leq {\sup_{\lrdb{u}\leq A, 0\leq q \leq T} \lrva{
				\mathscr{L}_{q}^{K_0}f\lr{u}-\mathscr{L}_{q}^{\infty}f\lr{u} }} \\
		&	\qquad\qquad\qquad +C\E\lr{\lrdb{U^\infty_{q}} \mathds{1}_{\lrdb{U^\infty_{q}}> A}}\\
		&\xrightarrow[K_0\to+\infty]{} C_{T}\E\lr{\lrdb{U^\infty_{q}} \mathds{1}_{\lrdb{U^\infty_{q}}> A}}\\
		&\xrightarrow[A\to+\infty]{} 0,
	\end{align*}
	with $C_{T}$ a constant which only depends on $T$, and where we have used Step~4 and the uniform integrability of $\sup_{0\leq q \leq T} \lrdb{U^\infty_{q}}$ from Step~2. This convergence being uniform in $r\in[0,T]$, we deduce that
	\begin{align}
		\label{convLK0LKinf} 
		\int_s^t \E\left({\lrva{
				\mathscr{L}_{r}^{K_0}f\lr{U_{r}}-\mathscr{L}_{r}^{\infty}f\lr{U_{r}} }}\right) \mathrm{d}r \xrightarrow[K_0 \to +\infty]{} 0 
	\end{align}
 and hence that
	\begin{align*}
		\Big|\E\Big(H&\lr{U^\infty}-H^{K_0}\lr{U^\infty}\Big)\Big|\xrightarrow[K_0 \to +\infty]{} 0 
	\end{align*}

	Let us now show that the second term~\eqref{eq:circle2} goes to  $0$ when $K\to+\infty$. 
	On the one hand, we note that $H^{K_0}$ is continuous for the norm $\lrdb{.}_{\infty}$ and, on the other hand, that, $U^\infty$ being almost surely continuous (the proof is identical to the one for the limit of  $\nu^K/K$ in the proof of Theorem~\ref{thm:LGN}), $U^K$ converges in law for the uniform norm on $[0,t]$ to $U^\infty$ (up to the extraction of a subsequence), and consequentially, using the uniform integrability of $H^{K_0}(U^K)$ (similarly as for the uniform integrability of $H^{K_0}(U^{K_0})$),
	\[
	\E\lrcr{H^{K_0}\lr{U}} \underset{K \to +\infty}{\longrightarrow} \E\lrcr{H^{K_0}\lr{U^{K}}}.
	\]
	This shows that the second term goes to $0$ when $K\to+\infty$.
	
	Let us finally show that the third term~\eqref{eq:circle3} converges to $0$ when $K_0$ and $K$ go to infinity. We have
	\begin{align*}
		\mathbb E\Big|H^{K_0}\lr{U^K}&-H^{K}\lr{U^{K}}\Big|\leq \lrdb{g_1}_{\infty} \cdots \lrdb{g_k}_{\infty}  \int_s^t \mathbb E\lrva{\mathscr{L}_{r}^{K}f\lr{U_{r}^K}-\mathscr{L}_{r}^{K_0}f\lr{U^K_{r}}}\mathrm{d}r\\
		&\leq \lrdb{g_1}_{\infty} \cdots \lrdb{g_k}_{\infty} \Bigg(  \int_s^t\E  \lrva{\mathscr{L}_{r}^{K}f\lr{U_{r}^K}-\mathscr{L}_{r}^{\infty}f\lr{U_{r}^K}}\mathrm{d}r\\
		&\qquad\qquad\qquad\qquad\qquad\qquad+ \int_s^t\E \lrva{\mathscr{L}_{r}^{\infty}f\lr{U_{r}^K}-\mathscr{L}_{r}^{K_0}f\lr{U^K_{r}}}\mathrm{d}r     \Bigg)\\
		&\xrightarrow[K,K_0\to+\infty]{} 0,
	\end{align*}
	according to \eqref{convLK0LKinf}.
	
	Finally, letting $K_0$ and $K$ go to infinity, we have shown that the limit process $U^\infty$ is a solution to the following martingale problem: for all function $f \in C_b^2\lr{\R^{L\times \lrva{E}},\R}$, the process
	\[
	f\lr{U^\infty_t}-f\lr{U^\infty_0}-\int_0^t \mathscr{L}_t^{\infty}f\lr{U^\infty_s} \mathrm{d}s \text{ is a martingale.}
	\]

	\textit{Step 6. Identification of the limit as the unique solution of a martingale problem.}
	Since we have
	\begin{align*}
		\mathscr{L}_t^{\infty}f(u) &= (Au)^{\top} \nabla f(u) 
		+\frac1{2}\, \text{Tr}(\sigma_t H_f(u) \sigma_t^{\top}),
	\end{align*}
	where we recall that $\sigma_t^{\top}\sigma_t=\sum_{e,y} \theta_e^{y}\nu_t(y) \bar C_e^{y}$, we deduce form~\cite[Section~6.5]{KS2014} that the martingale admits a unique solution $U$, which admits the representation
	\[
	U_t=e^{tA}U_0+\int_0^t e^{(t-s)A}\sigma_s\,\mathrm dW_s,
	\]
	where $W$ is standard $|X|$-dimensional Brownian motion. This concludes the proof of Theorem~\ref{thm:fluctuations}.
	\end{proof}

					\bibliographystyle{plain} 
					\bibliography{biblio}
					
				\end{document}